\documentclass[12pt]{amsart}

\usepackage[utf8]{inputenc}
\usepackage{amsmath}
\usepackage{amsfonts}
\usepackage{amssymb}
\usepackage{amsthm}
\usepackage{float}
\usepackage{mathtools}
\usepackage{csquotes}
\usepackage[numbers]{natbib}
\usepackage{pdfpages}
\usepackage{enumitem}

\RequirePackage{doi}

\usepackage[margin=3cm]{geometry}

\usepackage{tikz-cd}
\usetikzlibrary{arrows,matrix,shapes,decorations.text, mindmap,shapes.misc}
\usepackage{metalogo}

\newcommand{\Z}{\mathbb{Z}}
\newcommand{\N}{\mathbb{N}}

\renewcommand{\S}{\mathbb{S}}

\newcommand{\R}{\mathbb{R}}

\newcommand{\Tr}[0]{\text{Tr}}

\newtheorem{AThm}{Theorem}
\newtheorem{Thm}{Theorem} 

\newtheorem{Quest}{Question}
\newtheorem{Lem}[Thm]{Lemma}

\newtheorem{Prop}[Thm]{Proposition}
\usepackage{hyperref}
\theoremstyle{definition}
\newtheorem{Def}[Thm]{Definition}

\newtheorem{Rem}[Thm]{Remark}
\newtheorem{Ex}[Thm]{Example}
\newtheorem{assum}[Thm]{Assumption}

\usepackage[toc,page]{appendix}
\usepackage{dsfont}

\newcommand{\changelocaltocdepth}[1]{%
  \addtocontents{toc}{\protect\setcounter{tocdepth}{#1}}%
  \setcounter{tocdepth}{#1}%
}

\newcommand{\calH}{\mathcal{H}}
\newcommand{\calX}{\mathcal{X}}
\newcommand{\rad}{{\mathrm{rad}}}
\newcommand{\const}{{\mathrm{const}}}

\newcommand{\rdot}{\dot{r}}
\newcommand{\udot}{\dot{u}}
\newcommand{\xidot}{\dot{\xi}}

\newcommand{\etadot}{\dot{\eta}}
\newcommand{\gammadot}{\dot{\gamma}}
\newcommand{\thetadot}{\dot{\theta}}

\newcommand{\ydot}{\dot{y}}
\newcommand{\ytil}{\tilde{y}}

\newcommand{\Itilde}{\tilde{I}}
\newcommand{\ftilde}{{\tilde{f}}}

\newcommand{\Xtilde}{\tilde X}
\newcommand{\stilde}{\tilde s}
\newcommand{\gammabar}{\overline{\gamma}}

\newcommand{\eps}{\varepsilon}

\newcommand{\geucl}{g_{\mathrm{eucl}}}
\newcommand{\betaconic}{\beta_{\mathrm{conic}}}

\newcommand{\calN}{\mathcal{N}}
\newcommand{\dXtilde}{{\partial\Xtilde}}
\newcommand{\frakF}{\mathfrak{F}}
\newcommand{\frakFtilde}{\widetilde{\frakF}}

\newcommand{\expbar}{\overline{\exp}}

\newcommand\Tstrut{\rule{0pt}{2.6ex}}         

\makeatletter
\@namedef{subjclassname@2020}{%
  \textup{2020} Mathematics Subject Classification}
\makeatother

\graphicspath{{figures/}}

\title{Geodesics Orbiting a Singularity}


\author{Daniel Grieser}
\address{Mathematisches Institut, Universit\"at Oldenburg}
\email{daniel.grieser@uni-oldenburg.de}
\author{J\o rgen Olsen Lye}
\address{Institut für Differentialgeometrie, Leibniz Universit\"at Hannover}
\email{joergen.lye@math.uni-hannover.de}

\date{\today}
\begin{document}

\pagestyle{plain}

\begin{abstract}
We study the behaviour of geodesics on a Riemannian manifold near a generalized conical or cuspidal singularity. We show that geodesics 
entering a small neighbourhood of the singularity either hit the singularity or approach it to a smallest distance $\delta$ and then move away from it, winding around the singularity a number of times. We study the limiting behaviour $\delta\to0$ in the second case. In the cuspidal case the number of windings goes to infinity as $\delta\to0$, and we compute the precise asymptotic behaviour of this number.
 The asymptotics have explicitly given leading term determined by the warping factor that describes the type of cuspidal singularity. We also discuss in some detail the relation between differential and metric notions of conical and cuspidal singularities.
\end{abstract}

\maketitle

\phantomsection

\keywords

\tableofcontents

\section{Introduction}

Geodesics in Riemannian manifolds are among the most fundamental objects of differential geometry. Besides their intrinsic interest as locally shortest curves, or as trajectories of a free particle, they are used to define normal coordinates, which are of great utility. Also, they play an essential role in studying solutions of the wave equation, whose singularities  propagate along geodesics (see \cite{Hor}, for instance). 

In this paper we study the behaviour of geodesics near an isolated conical or cuspidal singularity of a Riemannian space. The precise definition of such singularities is given below, but for a first illustration of our results consider the example of the surface in $\R^3$ generated by rotating the curve $x=z^2$, $z\geq0$, around the $z$-axis, see the bottom left picture in Figure \ref{fig:ConeCuspPlot}. It has a cuspidal singularity at the origin $p=0$.
\begin{figure}
\includegraphics[width=.48\textwidth]{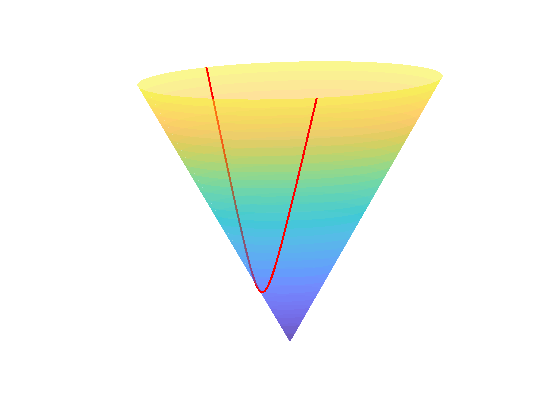}\hfill
    \includegraphics[width=.48\textwidth]{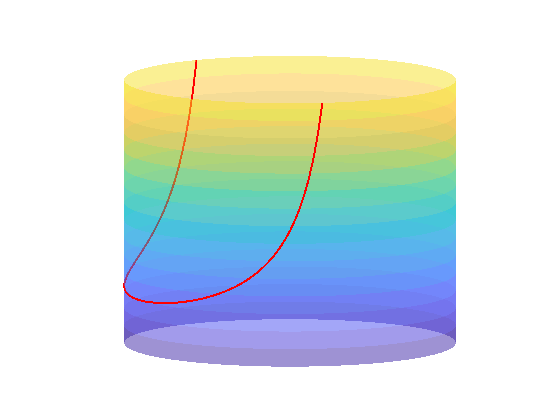}
    \\[\smallskipamount]
    \includegraphics[width=.48\textwidth]{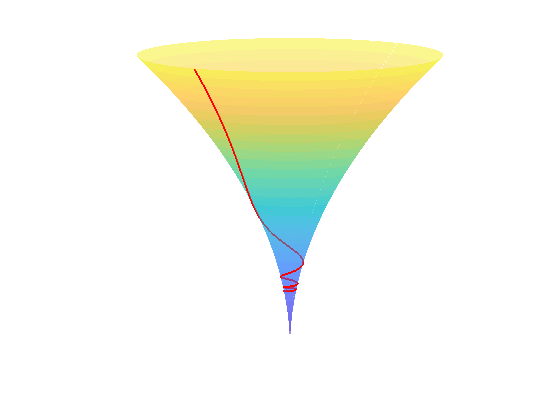}\hfill
    \includegraphics[width=.48\textwidth]{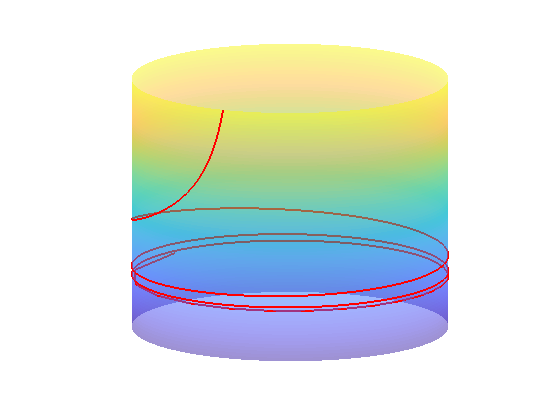} 
    \caption{Illustrations of Theorem \ref{thm:winding} for a conical metric $f(r)=r$ (above) and a cuspidal metric $f(r)=r^2$ (below). The right hand pictures show the product space $(0,R)\times Y$ while the left hand pictures illustrate the geometric situation. In all cases, $R=1.5$ and $\delta=0.3$. The $r$ resp.\ $z$ direction is upwards. Only the downward moving part of the geodesic is shown in the cuspidal case.}
    \label{fig:ConeCuspPlot}
\end{figure}

Consider the geodesics entering the neighbourhood $U=\{z<1\}$ of $p$ at any given point. 
It is quite obvious that one of these will hit $p$ after a finite time. Any other geodesic will move downward, reach a lowest point at $z=\delta$ say, then move up again, and finally leave $U$ (this can be seen using the classical Clairaut integral, for example, but we will prove it in greater generality). Inside $U$ this geodesic will wind around $p$ a number of times, as illustrated in Figure \ref{fig:ConeCuspPlot}.
In this special case, our results (see Theorem \ref{thm:winding}) imply that the number of windings is asymptotic
 to $\frac{\pi}{C} \delta^{-1}$ as $\delta\to0$, with $C=\int_{-\pi/2}^{\pi/2} \sqrt{\cos(\vartheta)}\, d\vartheta\approx 2.4$. 

Our setting and results are more general than this example in the following ways: we allow more general profile functions $x=f_0(z)$ instead of $z^2$, including $f_0(z)=z$ and $f_0(z)=e^{-1/z}$, for example. We also consider the natural generalization of rotation surfaces to product manifolds $(0,R)\times Y$ in any dimension with warped product metrics, where the cross section $Y$ is any closed Riemannian manifold (generalizing $Y=S^1$ in the surface case, where $f_0$ roughly corresponds to the warping function), and we also allow perturbations of such warped products. See below for the precise setting.

Our Main Theorems A, B, C are stated  in Subsection \ref{subsec:results} below. They
can be roughly summarized as follows.
\begin{Thm}[Rough summary of all results]
Consider a space with a conical or cuspidal singularity $p$, with cross section $Y$.
Any geodesic $\gamma$ in a neighbourhood of $p$ is either radial, i.e.\ it hits the singularity, or it approaches $p$, to a smallest distance $\delta$, and then moves away from $p$, while winding around $p$. For a winding geodesic $\gamma$ with small $\delta$ we have:
\begin{itemize}
 \item
The distance of $\gamma$ to $p$ behaves as for a radial geodesic up to an error of order $\delta$.
\item The $Y$-component of $\gamma$ closely follows a geodesic in $Y$. 
\item The length of this geodesic in $Y$ (generalizing the number of windings in the case $Y=S^1$) is asymptotic (as $\delta\to0$) to $C_f/f'(\delta)$, for a constant $C_f$ only depending on $f$.  
\end{itemize}
\end{Thm}

\subsection{Setting} \label{subsec:setting}
We now give a rough definition of the spaces that we consider in this paper. We refer to Section  \ref{Section:BasicGeom} for the precise definitions of a Riemannian space with isolated conical or cuspidal singularity along with technical assumptions. The key 
point is that the metric has generalized warped product form near the singularity, which we define as follows.

\begin{Def}
\label{def:intro warped product}
Let $Y$ be a compact manifold and $R>0$. A \textbf{generalized warped product metric} on $(0,R)\times Y$ is a smooth Riemannian metric of the form, with $r$ the coordinate on $(0,R)$,
\begin{equation}
\label{eq:Metric}
  g=dr^2+f(r)^2\, h_r\,,\qquad h_r = h(r,y)dy^2
\end{equation}
where $f:(0,R)\to(0,\infty)$ is continuously differentiable
and $h_r$ is a Riemannian metric\footnote{The notation $h(r,y)\,dy^2$ is supposed to be suggestive of the coordinate representation of $h_r$, which is $h_{ij}(r,y) \,dy^i dy^j$. We use the Einstein summation convention.} on $Y$ 
for each $r\in [0,R)$, depending continuously differentiably on $r$. The function $f$ is called the \textbf{warping function}.
If $h$ does not depend on $r$ then $g$ is called a \textbf{warped product metric}. 
\end{Def}

Of course, in \eqref{eq:Metric} the factor $f^2$ could be incorporated into $h$, but we will be interested in the case where $f(r)\to0$ as $r\to0$ while $h$ stays non-degenerate. 
In this case distances between points $(r,y)$ and $(r',y')$ tend to zero as $r\to0$, $r'\to0$, so geometrically one may complete the metric at $r=0$ by adding a single point, which we call the singularity. We call the resulting metric space including the singularity a \textbf{Riemannian space with isolated singularity}.
 Figure \ref{fig:ConeCuspPlot} illustrates the idea. 
\begin{Def}
\label{def:intro cone cups}
In the setting of Definition \ref{def:intro warped product}, assume
\begin{quote}
the function $f$ extends to a continuously differentiable map 
\[f:[0,R)\to[0,\infty)\]
 which satisfies 
\begin{equation}
 \label{eqn:props f}
f(0)=0\,,\quad f \text{ convex}.
\end{equation} 
\end{quote}
If $f'(0)>0$ then we speak of a \textbf{conical} singularity, while if $f'(0)=0$ then we speak of a \textbf{cuspidal} singularity. 
\end{Def}
\begin{Rem}
The convexity of $f$ is not needed in the conical case if $f$ extends to a $C^2$-function $f:[0,R)\to [0,\infty)$. See Remark \ref{rem:sing types} (a)-(c) below. See Remark \ref{rem:sing types} (d) for some typical examples of warping functions covered by this work.
\end{Rem}
 
\subsection{Results} \label{subsec:results}
We give a detailed description of the geodesics in the pointed neighbourhood $(0,R)\times Y$ of the conical/cuspidal singularity. 
We write geodesics as  
$$\gamma:I\to (0,R)\times Y\,,\qquad \gamma(t)=(r(t),y(t))\,,$$
$I\subset\R$ an interval,
and assume unit speed, $\vert \gammadot\vert_g\equiv 1$, i.e.\ $\rdot^2 + f(r)^2 |\ydot|^2 \equiv 1$, where a dot always indicates the derivative with respect to $t$ and $|\ydot|$ is the length with respect to $h_r$.

First, we have the following dichotomy.
\begin{AThm}[Radial and winding geodesics]
 \label{thm:two kinds}
 Let $g$ be a generalized warped product metric on $(0,R)\times Y$ as in \eqref{eq:Metric}. Then any geodesic is of one of the following types.
\begin{description}
 \item[radial] $y$ is constant, and $\rdot$ is constant equal to $1$ or $-1$.
 \item[winding] $\ydot(t)\neq0$ for all $t$.
\end{description}
If the warping function satisfies \eqref{eqn:props f} and if $R>0$ is small enough, then any maximal winding geodesic 
 $\gamma:I\to (0,R)\times Y$ satisfies in addition:
\begin{enumerate}[label=(\alph*)]
\item The function $t\mapsto r(t)$ is strictly convex.
 \item
$\gamma$ has finite length, i.e.\ $I=(T_-,T_+)$ with $T_\pm\in\R$.
 \item $\gamma$ enters and leaves at $r=R$,  i.e.  $r(t)\to R$ as $t\to T_\pm$.
\end{enumerate} 
In the warped product case, $t\mapsto y(t)$ is a time reparametrised geodesic of $(Y,h)$.
\end{AThm}
In the sequel, all geodesics will be assumed to be maximal, i.e.\ their domain $I$ cannot be enlarged.

We now give a more precise description of the  winding geodesics $\gamma=(r,y)$. By (a) and (c) above, $r$ assumes its minimum at a unique time, which we may and will always take to be $t=0$. We write $\delta=r(0)$ for the minimum. Since geodesics are uniquely determined by their initial point,
we get a parametrisation
\begin{equation}
\label{eqn:winding geod param} 
 (0,R) \times SY \to \{\text{winding geodesics}\}\,,\quad (\delta, y_0,v_0)\mapsto \gamma_{\delta,y_0,v_0}
\end{equation}
where $SY=\{(y,v)\in TY\,:\ |v|_{h_0}=1\}$ is the unit tangent bundle and $\gamma_{\delta,y_0,v_0}$ is the maximal geodesic starting at $(\delta,y_0)$ in direction $(0,v_0)\in T_{(\delta,y_0)}((0,R)\times Y)$.

The next two theorems describe the behaviour of  winding geodesics as $\delta\to0$. Theorem \ref{thm:radialpart} describes how the radial component $r_{\delta}$ behaves, including a comparison theorem. Theorem \ref{thm:winding} describes the $Y$ component $y_{\delta}$, in particular the asymptotics of its length.
\begin{AThm}[Radial component of geodesics]
 \label{thm:radialpart}
 Let $(X,d)$ be a Riemannian space with an isolated conical or cuspidal singularity, and assume $R$ is small enough so \eqref{eq:Rc} is satisfied.
\begin{enumerate}[label=(\alph*)]
 \item Let $\gamma_\delta = (r_\delta,y_\delta):I_\delta\to(0,R)\times Y$, $\delta\in(0,R)$, be any family of unit speed geodesics, with $\min_t r_\delta(t)=r_\delta(0)=\delta>0$, with maximal interval of existence $I_\delta$. 
 
 As $\delta\to 0$ we have $I_\delta\to(-R,R)$ (i.e.\ the endpoints of $I_\delta$ converge to $\pm R$)
   and 
 \[ r_\delta(t) \to |t| \text{ for all } t\in (-R,R)\,.\]
\item In the case of a warped product with convex warping function $f$, the following comparison principle holds. If $\gamma$, $\overline{\gamma}$ are two unit speed geodesics which both reach their lowest point at $t=0$, and if $r(0)<\overline{r}(0)$, then $r(t)<\overline{r}(t)$ for all time $t$. 
\end{enumerate}
 \end{AThm}
 In fact, we have precise error estimates for the convergence in (a), see Lemma \ref{Lem:rBounds}. Note that combining (a) and (b) we obtain
 $$ r(t)> |t| \text{ for all }t $$
 in the warped product case, for any winding geodesic. This is not obvious even for a conical surface.
\medskip
 
 \begin{AThm}[Angular component of geodesics]
\label{thm:winding}
Assume the setup of Theorem \ref{thm:radialpart}, and let $\gamma_\delta=(r_\delta,y_\delta)$ be a family of geodesics as in part (a) of that theorem.
Assume that  the warping function $f$ satisfies the non-oscillation condition \eqref{eqn:F condition}.
\begin{enumerate}[label=(\alph*)]
 \item The length of the $Y$-projection of $\gamma_\delta$,
\[\ell(y_{\delta})\coloneqq \int_{I_\delta} \vert \ydot_{\delta}(t)\vert_{h}\, dt\] 
satisfies
\begin{equation}
 \label{eqn:bd length asymp}
 \ell(y_{\delta})\sim  \frac{C_f}{f'(\delta)},
\end{equation}
where $\sim$ means that the quotient of the left hand side by the right hand side tends to one as $\delta \to 0$ and
\[C_f\coloneqq \int_{-\frac{\pi}{2}}^{\frac{\pi}{2}} \mathfrak{F}\left(\frac{1}{\cos(\vartheta)}\right)\, d\vartheta\]
with $\mathfrak{F}$ defined in \eqref{eqn:F condition}. The constant $C_f$ satisfies 
\[2\leq C_f\leq \pi,\]
with $C_f=\pi$ in the conical case and $C_f\to 2$ for $f(x)=x^{\alpha}$ as $\alpha\to \infty$. The sharper the cusp, the smaller is $C_f$ (see \eqref{eqn:frakF comparison} for a precise statement). 
\item 
Suppose that $y_\delta(0)\to y_0$, $f(\delta)\ydot_\delta(0)\to v_0$ as $\delta\to0$ for some $(y_0,v_0)\in SY$. Let $\ytil:\Itilde\to Y$ be the unit speed geodesic in $(Y,h_0)$ satisfying $\ytil(0)=y_0$, $\dot\ytil(0)=v_0$, where
$$ \Itilde=(-\frac{\pi}{2},\frac{\pi}{2}) \text{ in the conical case with } f'(0)=1, \quad \Itilde=\R \text{ in the cuspidal case.}$$
Then $y_\delta$ converges after unit speed reparametrisation to $\ytil$, uniformly on compact subsets of $\Itilde$.
\end{enumerate}
\end{AThm}

More explicitly, the statement in (b) is: Define the rescaled time  $\tau$ by $\tau(0)=0$ and
\begin{equation}
 \frac{d\tau}{dt} = |\ydot_\delta|
\label{eq:dtau}
\end{equation}
Then for $\Itilde_\delta=\tau(I_\delta)$ and $\ytil_\delta:\Itilde_\delta\to Y$ 
defined by 
$\tilde{y}_{\delta}(\tau(t))=y_{\delta}(t)$ we have, as $\delta\to0$,
$$ \Itilde_\delta\to \Itilde\,, \quad\text{and}\quad \ytil_\delta \to \ytil \text{ uniformly on compact subsets of $\Itilde$.
}$$
The asymptotic length in \eqref{eqn:bd length asymp} behaves as follows, as $\delta\to0$:
$$ \frac{C_f}{f'(\delta)}\to \frac{\pi}{f'(0)} \ \text{ (conical case)},\quad \frac{C_f}{f'(\delta)}\to\infty \ \text{ (cuspidal case).}$$
This shows that parts (a) and (b) of the theorem are consistent, since $|\Itilde_\delta|=\ell(\gamma_\delta)$.

In particular, Theorem \ref{thm:winding}(b) implies that every compact segment of a geodesic on $Y$ (having length less than $\pi/f'(0)$ in the conical case) is the uniform limit of the $Y$-parts of some family $(\gamma_\delta)$ of geodesics on $X$.
\begin{Rem} \label{rem:main theorems}
\begin{enumerate}
 \item It is worth stressing that the constant on the right hand side of \eqref{eqn:bd length asymp} only depends on $f$, and not on $h$, the dimension, or any detailed description of $\gamma_{\delta}$. We discuss the function $\mathfrak{F}$ and the constant $C_f$ in more detail in Section \ref{Section:Discussion}, including bounds and several examples.
 \item In the case of $Y=S^1$ of length $2\pi$ Equation \eqref{eqn:bd length asymp} means that a  geodesic which almost hits the singularity winds roughly $\frac{C_f}{2\pi}\cdot \frac{1}{f'(\delta)}$ times around the singularity before leaving. See Figure \ref{fig:ConeCuspPlot} for a couple of illustrations.
\item The assumed convergence of the initial conditions of $y_\delta$ in part (b) of Theorem \ref{thm:winding} is always satisfied after passing to a subsequence, by compactness. 
\end{enumerate}
\end{Rem}

\subsection{Methods, outline of proofs}
We formulate the geodesic equations as a Hamiltonian system of first order ODEs. In the warped product case, the equations for $r$ completely decouple and depend only on $f$. The motion in the $r$-direction then determines the speed in the $Y$-direction via the unit speed condition. This allows us to study the asymptotic behaviour quite explicitly. In the generalized warped case, the equations of motion almost decouple. In particular, the leading order behaviour turns out to be the same as in the warped case, and we derive bounds which imply that the generalized warped case has the same asymptotic behaviour as the warped one.

Perhaps the most surprising of our results is the length asymptotics \eqref{eqn:bd length asymp}. Here is a quick rough explanation for the appearance of the term $f'(\delta)$, in the case of a cuspidal warped product metric: Denote by $\theta\in(-\frac\pi2,\frac\pi2)$ the angle between $\gammadot$ and the \lq circle\rq\ of latitude $r=$ const, for a given geodesic $\gamma=\gamma_\delta$. 
Then $|\gammadot|=1$ implies $\cos\theta=f(r)|\ydot|$, $\sin\theta=\rdot$. 
The well-known Clairaut relation (which we rederive) says that $f(r)\cos\theta$ is constant along $\gamma$. Differentiating it in $t$ we obtain $\thetadot=f'(r)|\ydot|$, or $|\ydot|dt = \frac1{f'(r)}d\theta$. The length of $\gamma$ is obtained by integrating this, and \eqref{eqn:bd length asymp} essentially follows from the fact that $r$ is on the order of $\delta$ when $\theta$ is bounded away from $\pi/2$.

The structure of the paper is as follows. We start by describing the geometry of the spaces we work with in some more detail in Section \ref{Section:BasicGeom}. We proceed by stating the geodesic equations as a Hamiltonian system in Section \ref{Section:Hamilton}. Here we also introduce several useful variables and deduce a Clairaut-like relation. Section \ref{Section:Winding} is the heart of the paper, where we analyse the equations of motion and prove Theorem \ref{thm:two kinds}, \ref{thm:radialpart} and  \ref{thm:winding}.  Section \ref{Section:Discussion} deals with the constant $C_f$ appearing in the length asymptotics \eqref{eqn:bd length asymp}, including explicit computations of it and the necessity of assuming convexity and that \eqref{eqn:F condition} exists. Section \ref{sec:singularities} relates our notion of conical/cuspidal spaces to metric-free definitions of such spaces. Here we also present some families of spaces for which our assumptions hold.

\subsection{Related work and further remarks}\label{subsec:rel work}
Our results are new only in the cuspidal case but give a unified treatment of conical and cuspidal case. The conical case was treated by Melrose and Wunsch in \cite[Definition 1.4, Lemma 1.5]{MeWu:Geo} in the context of wave propagation, after early work by Stone \cite{Sto:EMISP}.
The first author discusses conical metrics in detail in \cite{Gri:NDOCS}.

The family of geodesics hitting the singularity was analysed for a more general class of $k$-cuspidal (i.e.\ $f(r)=r^k$, $k\geq2$) metrics -- those arising from differential $k$-cuspidal singularities, see Subsection \ref{subsec:bl-up} and in particular \eqref{eqn:cusp metric general}  -- by Grandjean and the first author \cite{GrGr18}. By Theorem \ref{thm:two kinds}, the geodesics hitting the singularity are simply given by constant $y$, for generalized warped product metrics. In \cite{GrGr18} it is shown that metrics on $k$-cuspidal singularities are not of generalized warped product form (see Subsection 6.2 and in particular \eqref{eqn:cond warped cusp} in the present work), and that the family of geodesics hitting the singularity has a more complicated structure in this setting. 

Previously, Bernig and Lytchak \cite{BerLyt:TSGHLSS} obtained first order information for geodesics on general real algebraic sets $X\subset\R^n$, by showing that any geodesic reaching a singular point $p$ of $X$ in finite time must have a limit direction at $p$.

For more on warped products and geodesics, one can consult \cite[Section 7]{Negative} and \cite[Chapter 7]{SemiRiemann}.

\section{Generalized Warped Product Geometry}
\label{Section:BasicGeom}
We start by being a bit more precise about the family of metrics $h_r$ in our definition of generalized warped products. 

\begin{assum}[Regularity assumption] \label{assump:reg} The metrics $h_r$ in the generalized warped product metric of Definition \ref{def:intro warped product} are uniformly positive definite and depend uniformly $C^1$ on $r$. These uniformity conditions are equivalent to 
\begin{equation}
-2c h \leq \partial_r h \leq 2c h
\label{eq:drhBound}
\end{equation}
for all $r\in (0,R)$, for some constant $c\geq0$. 
\end{assum}
This implies that the family extends to a continuous family of metrics on $r\in[0,R]$.

Next, we introduce singular Riemannian spaces whose singularities are modelled on generalized warped products as defined above.
\begin{Def} \label{def:cone cusp}
A metric space $(X,d)$ is called a \textbf{Riemannian space with an isolated conical or cuspidal singularity at $p\in X$} if
\begin{itemize}
 \item
 $X\setminus\{p\}$ is a smooth manifold, and the metric $d$ is induced by a Riemannian metric $g_X$ on 
$X\setminus\{p\}$ ;
 \item there is a neighbourhood $U$ of $p$ in $X$, a number $R>0$ and a compact manifold $Y$, and a continuous map
 $$ \beta:[0,R) \times Y \to U $$
which sends $\{0\}\times Y$ to $p$ and restricts to an isometry
$\bigl((0,R) \times Y, g\bigr) \to  \bigl(U\setminus\{p\},g_X\bigr)$
 where $g$ is a generalized warped product metric as in \eqref{eq:Metric} and where $f$ satisfies the conditions of Definition \ref{def:intro cone cups}.
As in that definition, if $f'(0)>0$ then we speak of a \textbf{conical} singularity, while if $f'(0)=0$ then we speak of a \textbf{cuspidal} singularity. 
\end{itemize}

\end{Def}
See Section \ref{sec:singularities} for examples, in particular for the proof that surfaces of revolution as indicated above fit into this framework, and for the relation of this notion of conical/cuspidal singularity to other natural such notions.
\begin{Rem}
\mbox{}
 \label{rem:sing types}
\begin{enumerate}[label=(\alph*)]
 \item
 As remarked above, the warping function $f$ is not uniquely determined by the metric $g$: Replacing $f$ by $af$ and $h$ by $a^{-2}h$, where $a$ is positive and $C^1$ on $[0,R)$, yields the same metric. This is the only freedom, by the conditions on $h$. Thus, if we call two warping functions equivalent if they differ by such a factor $a$ then $g$ determines the equivalence class $[f]$ of $f$, and it makes sense to speak of a \textbf{singularity of type $[f]$}.

 \item
 If $f$ vanishes to finite order $k$ at zero and is $C^{k+1}$ on $[0,R)$ then it is equivalent to $f_k(r)=r^k$.
\item
We use convexity of $f$ in several arguments. 
However, if $f$ is non-convex but there is  an equivalent function $af$ which is convex
then it follows that Theorems \ref{thm:two kinds}, \ref{thm:radialpart}(a) hold verbatim and Theorem \ref{thm:winding} holds with $f$ replaced by $af$. For instance, this is true in the $C^2$ conical case.
Note that convexity of $f$ implies $f'(r)>0$ for $r>0$.

The example of Section \ref{subsec:why convexity} demonstrates that the asymptotics of Theorem C can change without convexity. Convexity of $f$ implies that the angle $\theta$ between a geodesic and the level sets $\{r\}\times Y$ increases in $r$, see \eqref{eqn:thetadot}.
See also Remark \ref{Rem:MeanConvex} for an interpretation of the convexity of $f$ in terms of curvature. 

\item Some classes of functions $f$ to have in mind are\footnote{The first and second family overlap when $\alpha=1=\mu$. All three families are convex for small values of $x$. More precisely, the second family is convex as long as $y=\ln\left(\frac{1}{x}\right)$ satisfies $\mu y^\mu-y-(\mu-1)\geq 0$. The third family is convex when $\sqrt[\beta]{\frac{\beta}{\beta+1}}\geq x$. Both can be achieved by shrinking $R$.} 
\begin{itemize}
\item $f(x)=x^{\alpha}$ for $\alpha\in [1,\infty)$. 
\item $f(x)=\exp\left(-\ln\left(\frac{1}{x}\right)^{\mu}\right)$ for $\mu\in [1,\infty)$.
\item $f(x)= \exp\left(-\frac{1}{x^\beta}\right)$ for $\beta\in (0,\infty)$. 
\end{itemize}
 \end{enumerate}
\end{Rem}

We will identify $U\setminus\{p\}$ with $(0,R)\times Y$. The map $\beta$ should be thought of as a generalization of polar coordinates: For $X=\R^2$ 
the map
$$\beta:[0,\infty)\times S^1\to \R^2\,,\quad (r,\phi)\mapsto (r\cos\phi,r\sin\phi) $$
is an isometry over $r>0$ for the metric $dr^2+r^2d\phi^2$ on $(0,\infty)\times S^1$ (where $S^1=\R/2\pi\Z$) and the Euclidean metric on $\R^2\setminus\{0\}$.
So $0\in\R^2$ can be considered a conical singularity in this sense, and the same is true for any smooth point of a Riemannian manifold (with $Y$ equal to the sphere). The space $[0,\infty)\times Y$ is sometimes called the \textbf{blow-up} of $X$ in $p$, and the map $\beta$ the \textbf{blow-down} map.

Note that, while $g$ is a Riemannian metric  in $r>0$, it is only positive semi-definite at $r=0$, with any two points at $r=0$ having distance zero with respect to $g$. This reflects the fact that the map $\beta$ crunches all points at $r=0$ to the single point $p$. The warping function determines the 'speed' at which the crunching happens as $r\to0$.
Also, note that the form \eqref{eq:Metric} of the metric implies that $r$ is the distance to the singularity $p$.

\begin{Rem} \label{rem:correct def}
There is no unique answer to the question what the 'correct' definition of a Riemannian space with an isolated singularity is. While our notion of singularity is quite general in terms of $f$ and $h_r$, it is somewhat restrictive in requiring that there are no mixed terms in the metric \eqref{eq:Metric}, i.e.\ that coordinates $r,y$ can be chosen so that the lines $y=$ const are perpendicular to the hypersurfaces $r=$ const. See Section \ref{section: non-warped example} for a natural example where this is not satisfied.
We leave it for future work to analyse the geodesic flow for some of these more general metrics.

\end{Rem}
For Theorem \ref{thm:winding}, we need $f$ to satisfy the following condition:

\begin{assum}[Non-oscillation condition]
\label{ass:non-oscill}
 For the inverse function $F\coloneqq f^{-1}$ the limit
\begin{equation}
  \label{eqn:F condition}
\mathfrak{F}(\sigma)\coloneqq \lim_{\epsilon\to 0} \frac{F'(\sigma \epsilon)}{F'(\epsilon)}
\end{equation}
 exists  for $\sigma\in [1,\infty)$, and the convergence is uniform on compact subsets.
\end{assum}

This condition is satisfied for the examples above. 
See Lemma \ref{Lem:FNonConv} for a family of cuspidal examples where it is not satisfied. Furthermore, we need
(except in Section \ref{Section:Hamilton})  to limit the variation of $h_r$ over the interval $(0,R)$:
\begin{assum}[Small perturbation condition]
The constant $c$ in  \eqref{eq:drhBound} can be chosen so that
\begin{equation}
Rc<1. 
\label{eq:Rc}
\end{equation}
\end{assum}
Of course this can be achieved by shrinking $R$.
The condition \eqref{eq:Rc} is used to prove \eqref{eqn:thetadot} below, which in turn gets used implicitly several times.

\begin{Rem}
\label{Rem:MeanConvex}
The small perturbation condition along with the convexity of $f$ has an interpretation in terms of the mean curvature of the level sets of $r$:
A standard computation shows that the mean curvature vector of
$\{r\}\times Y\subset (0,R)\times Y$  is $H=H_0\partial_r$ where
\[H_0 =-\dim(Y)\frac{f'}{f} - \frac{1}{2} \Tr_h(\partial_r h)\,.\]
By the bound \eqref{eq:drhBound}, we can bound  $H_0$ as
\[-\dim(Y)\left(\frac{f'}{f}+c\right)\leq H_0\leq -\dim(Y)\left(\frac{f'}{f}-c\right)\]
The convexity of $f$ and $f(0)=0$ imply $\frac{f'(r)}{f(r)}\geq \frac{1}{r}\geq \frac{1}{R}$ (see Lemma \ref{Lem:f'fEstimate}), so the small perturbation condition implies a negative upper bound for the scalar mean curvature, 
\[H_0\leq -\frac{\dim(Y)}{R}(1-Rc)<0.\]
We will not use this geometric interpretation directly.
\end{Rem}

\section{Geodesic equations}
\label{Section:Hamilton}
In this section we analyse the geodesic equations for generalized warped product metrics of the type \eqref{eq:Metric}, but without extra conditions on the warping function $f$.
\subsection{Review of the Hamiltonian approach}
We use the Hamiltonian description of the geodesic flow. Recall that this means that (constant speed) geodesics on a Riemannian manifold $(Y,h)$ are the projections to $Y$ of curves in the cotangent bundle $T^*Y$, which are the integral curves of the Hamiltonian vector field for the Hamilton function
$\calH_h:T^*Y\to\R$,
$$ \calH_h(y,\eta) = \frac12 |\eta|^2_{h_y} \,,\quad \eta \in T^*_yY $$
where by
$h_y$ we also denote the metric on $T^*_yY$ dual to the metric $h_y$ on $T_yY$. In coordinates, $|\eta|^2_{h_y} = h^{ij}(y)\eta_i\eta_j$ where $(h^{ij})$ is the inverse matrix of $(h_{ij})$. Often we write simply $|\eta|_h$ or $|\eta|$.
The Hamilton vector field is, in coordinates,
$$ \calX_h = \frac{\partial \calH_h}{\partial\eta}\partial_y - \frac{\partial\calH_h}{\partial y} \partial_\eta $$
so its integral curves, also called \textbf{lifted geodesics}, are solutions $t\mapsto (y(t),\eta(t))$ of the system of differential equations
$$ \ydot = \frac{\partial \calH_h}{\partial\eta}(y,\eta)\qquad \etadot = - \frac{\partial\calH_h}{\partial y} (y,\eta) \,.$$
Geodesics are then the $y(t)$ parts of solutions of this system.
It is a basic fact that the Hamiltonian function is constant along integral curves, i.e. $\calH_h(y(t),\eta(t))=$ const. 
Note that 
$$ \frac{\partial \calH_h}{\partial\eta} (y,\eta) = \eta^\sharp $$
where $\eta\mapsto\eta^\sharp$, $T^*Y\to TY$ is the isomorphism induced by the metric.
In coordinates, $(\eta^\sharp)^i =h^{ij} \eta_j$. 
So for a geodesic $t\mapsto y(t)$ its lift is the curve $(y(t),\eta(t))$ where $\eta^\sharp(t)=\ydot(t)$, and this explains the relationship of the Hamiltonian approach to geodesics (with lift in the cotangent bundle) to the more standard approach where the lift is the curve $(y(t),\ydot(t))$ in the tangent bundle.

\medskip

\subsection{Geodesic equations for generalized warped product metrics}
We apply this general discussion with $Y$ replaced by $(0,R)\times Y$, with the metric $g$ in \eqref{eq:Metric}. Points are denoted $(r,y)$, and the dual cotangent variables are denoted $(\xi,\eta)$, so $\xi\in\R$, $\eta\in\R^{n-1}$ in coordinates.
The Hamilton function is 
\begin{equation}
\mathcal{H}_g=\frac12\left(\xi^2 +\frac{\vert \eta\vert^2_h}{f(r)^2}\right)
\label{eq:Energy}
\end{equation}
As mentioned above, $\calH_g$ is constant along integral curves.
We will always consider unit speed geodesics, i.e.\ integral curves lying on the hypersurface $\calH_g=\frac12$, or
\begin{equation}
\label{eqn:constant energy}
1 = \xi^2 +  \frac{\vert \eta\vert^2_h}{f(r)^2}.
\end{equation}
We calculate the Hamilton vector field for $\calH_g$. Since $\calH_g$ depends on $y,\eta$ only through $|\eta|^2_h$ we have
\begin{equation} \label{eqn:geod field}
 \calX_g = \calX_\rad + \frac1{f(r)^2} \calX_h
\end{equation}
where $\calX_\rad= \frac{\partial \calH_g}{\partial\xi}\partial_r - \frac{\partial \calH_g}{\partial r}\partial_\xi$ governs the radial motion and $\calX_h$ is the Hamilton vector field of $\calH_h$, hence governs the motion in the $Y$ directions. Recall that the latter depends 
parametrically on $r$ since $h$ does, and as such defines a family of vector fields on $Y$.

We first observe that $\calX_g$ is tangential to the submanifold $\{\eta=0\}$ since the $\partial_\eta$ coefficient, which is $-f(r)^{-2}\partial |\eta|^2_{h_y} / \partial y$, vanishes at $\eta=0$. This implies that any integral curve of $\calX_g$ satisfies
$$ \text{either $\eta(t)=0$ for all $t$ or $\eta(t)\neq0$ for all $t$.} $$
The Hamilton equations for $\rdot$ and $\ydot$ are
$$ \rdot = \frac{\partial \calH_g}{\partial\xi} = \xi\,,\quad \ydot = \frac{\partial \calH_g}{\partial\eta} = \frac1{f(r)^2}\frac{\partial \calH_h}{\partial\eta} = \frac1{f(r)^2}\eta^\sharp\,.$$
From this we deduce the first part of Theorem \ref{thm:two kinds}: Geodesics with $\eta\equiv0$ have constant $y$, and $\rdot=\xi\equiv \pm1$ by the unit speed condition \eqref{eqn:constant energy}, so they are radial. All other geodesics have $\ydot(t)\neq0$ for all $t$, so they are winding.

We now consider winding geodesics, i.e. $\eta(t)\neq0$ for all $t$, and derive the full Hamilton equations for them.
Calculating $\frac{\partial \calH_g}{\partial r} = \frac{\vert \eta\vert^2_h}{f(r)^2} \left(-\frac{f'(r)}{f(r)} + \frac{\partial_r|\eta|}{|\eta|}\right)$ and using \eqref{eqn:constant energy} we get

\begin{equation}
 \calX_\rad = \xi \partial_r + (1-\xi^2)\left(\frac{f'(r)}{f(r)} - \tfrac{\partial_r|\eta|}{|\eta|}\right) \partial_\xi
\end{equation}
We write $\tfrac{\partial_r|\eta|}{|\eta|}$ in small font to emphasize that it is to be considered as a small perturbation (it vanishes in the warped product case).
Correspondingly, the lifted geodesics are solutions $t\mapsto (r(t),y(t),\xi(t),\eta(t))$ of the system
\begin{align}
\rdot &=\xi & 
\xidot &=(1-\xi^2)\left(\frac{f'(r)}{f(r)}-\tfrac{\partial_r|\eta|}{|\eta|}\right) 
 \label{eqn:ham r xi}\\[3mm]
\ydot &=f(r)^{-2} \,\partial_\eta \calH_h&
\etadot &=-f(r)^{-2}\, \partial_y \calH_h
 \label{eqn:ham phi eta}
\end{align}
By \eqref{eqn:constant energy}, \eqref{eqn:props f}, and the assumption $\eta(t)\neq 0$ for all $t$, the variable $\xi$ is constrained to lie in $(-1,1)$, so we may introduce a new variable $\theta$ (i.e.\ coordinate on the hypersurface $\calH_g=\frac12$)  by
$$ \xi = \sin \theta\,,\qquad \theta\in\left(-\frac\pi2,\frac\pi2\right)\,.$$
This will simplify the calculations below. The equations \eqref{eqn:ham r xi}  then turn into
\begin{align}
 \label{eqn:ham r theta}\tag{\ref{eqn:ham r xi}'}
 \rdot &= \sin\theta &
 \thetadot =  \left(\frac{f'(r)}{f(r)} - \tfrac{\partial_r|\eta|}{|\eta|}\right) \cos\theta
\end{align}
and the unit speed condition \eqref{eqn:constant energy} turns into
\begin{equation}
 \label{eqn:unit speed theta}
 f(r) \cos\theta = |\eta|\,.
\end{equation}
Note that $\rdot=\sin\theta$, $|\gammadot|=1$ implies that $\theta$ is the angle between $\gammadot$ and the 'circle of latitude' $\{r\}\times Y$. 
\subsection{The warped product case; Clairaut's integral}
\label{Section:Clairaut}
In the warped product case, i.e.\ where $h$ does not depend on $r$, two simplifications happen: The vector field $\calX_\rad$ is independent of $\eta$ since the term $\partial_r|\eta|$ vanishes, and the vector field $\calX_h$ is independent of $r$.  

This means that lifted geodesics $t\mapsto (r(t),\theta(t),y(t),\eta(t))$ are given as follows:
\begin{enumerate}
 \item $(r(t),\theta(t))$ solves the system \eqref{eqn:ham r theta}
 \item $(y(t),\eta(t))$ is an integral curve for the Hamilton vector field $\calX_h$, but with time reparametrised using the factor $f(r)^{-2}$ with $r=r(t)$ obtained in step 1. 
\end{enumerate}
Here we use that the vector field $\calH_h$ and the time-dependent vector field $\frac1{f(r(t))^2}\calH_h$ on $Y$ have the same integral curves except for time reparametrisation.
We also see that, along each integral curve,
\begin{equation}
\label{eqn:Clairaut} 
f(r)\cos\theta = \const\,,
\end{equation}
either by direct calculation from \eqref{eqn:ham r theta} or using \eqref{eqn:unit speed theta} and the fact that $|\eta|$ is constant along integral curves for $\calX_h$. The relation \eqref{eqn:Clairaut} completely determines the solutions of the $r,\theta$ system up to time parametrisation.

The expression on the left of  \eqref{eqn:Clairaut} is known as  Clairaut's integral in the context of surfaces of revolution. 

\subsection{Estimates for the deviation from the warped product case}
If $h$ depends on $r$ then the \lq Clairaut integral \rq\ $f(r)\cos\theta$ is not constant along geodesics, but we can estimate its variation. 

By \eqref{eqn:unit speed theta} the Clairaut integral equals $|\eta|$, so we consider this quantity. Note that $|\eta|=|\eta|_h$ depends on $r$ through $h$.
 
We recall that Assumption \ref{assump:reg} says
\begin{equation}
\tag{\ref{eq:drhBound}'}
-2ch \leq \partial_r h \leq 2ch
\end{equation}
for a constant $c\geq0$.
\begin{Lem} \label{lem:gen warp estimates}
With $c$ as above, we have for any  $\eta\in T^*Y\setminus\{0\}$:
\begin{gather}
-c\leq \tfrac{\partial_r \vert \eta\vert}{\vert \eta\vert} \leq c\,.
 \label{eqn:error est}
\end{gather} 
Furthermore, along any lifted geodesic we have
\begin{equation} \label{eqn:eta variation est}
\left| \frac d{dt}|\eta| \right| \leq c|\eta|\, \vert\sin\theta \vert
\end{equation}
\end{Lem}
Estimate \eqref{eqn:eta variation est} quantifies the fact that lifted geodesics are tangent to $\{\eta=0\}$.
\begin{proof}
Let $h^{-1}$ denote the inverse of $h$ considered as a matrix in some local coordinates. 
We multiply (\ref{eq:drhBound}') by $h^{-1}$ from left and right to get  
$-2ch^{-1} \leq \partial_r h^{-1} \leq 2ch^{-1}$ and therefore $-2c\vert \eta\vert^2 \leq \partial_r \vert \eta\vert^2\leq 2c\vert \eta\vert^2$. This implies \eqref{eqn:error est}.

Lifted geodesics are integral curves of $\calH_g$, and along such a curve we have
\begin{align*}
 \frac d{dt}\calH_h &= (\partial_y \calH_h) \ydot + (\partial_\eta \calH_h) \etadot + (\partial_r \calH_h) \rdot
 \\
 & = (\partial_r \calH_h) \rdot
\end{align*}
using \eqref{eqn:ham phi eta}, and this gives
$ \frac d{dt}|\eta| = \partial_r|\eta| \sin\theta$ from which we deduce \eqref{eqn:eta variation est} using \eqref{eqn:error est}.
\end{proof}

\section{Proofs of the main theorems}
\label{Section:Winding}

In this section we use the results from the previous section in combination with the properties \eqref{eqn:props f} of $f$ to prove Theorems \ref{thm:two kinds}, \ref{thm:radialpart}, and \ref{thm:winding}.

First, we supplement the estimates in Lemma \ref{lem:gen warp estimates} by estimates that use the special properties of $f$.
\begin{Lem}
\label{Lem:f'fEstimate}
 A warping function $f$ as in \eqref{eqn:props f} satisfies
\begin{equation}
  \frac{f'(r)}{f(r)} \geq \frac1r 
 \label{eqn:f' f estimate} 
\end{equation}
Furthermore, if $R<\frac1c$ (with $c$ satisfying \eqref{eq:drhBound}) then along any lifted winding geodesic we have
\begin{equation}
\label{eqn:thetadot} 
 \thetadot > 0
\end{equation}
\end{Lem}
\begin{proof}
  \eqref{eqn:f' f estimate} follows from $f(0)=0$ and the convexity of $f$:
 $ f(r)=\int_0^r f'(s)\,ds \leq r f'(r)$. 
 
 The equation \eqref{eqn:ham r theta} for $\thetadot$ along with \eqref{eqn:f' f estimate} and \eqref{eqn:error est} imply $\thetadot \geq (\frac1r-c)\cos\theta$. Also $\cos\theta>0$ for winding geodesics, which gives the second claim.
\end{proof}

\begin{proof}[Proof of Theorem \ref{thm:two kinds}]
By the Clairaut-like relation \eqref{eqn:unit speed theta}, we see that $\cos\theta\neq 0$ for winding geodesics. From $\rdot=\sin\theta$ and \eqref{eqn:thetadot}  we get 
\[\ddot{r}=\thetadot\cos\theta> 0\,,\]
so $r$ is strictly convex.
Then (b), (c) follow easily since geodesics exist as long as $r<R$.

The statement about reparametrisation in the warped product case was explained in Section \ref{Section:Clairaut}.
\end{proof}

For the proof of Theorem \ref{thm:radialpart} 
it is useful to introduce two new variables $\rho$ and $u$ (i.e. functions on $(0,R)\times Y$, resp.\ its cotangent bundle), defined by the identities: 
\begin{equation}
 \label{eqn:def rho u}
\begin{aligned}
  \rho &= f(r) \\
  \rho \sin \theta &= f(u)
\end{aligned}
\end{equation}

We will use the following bounds on the variation of $\eta$ and $r$  along a geodesic. 
 \begin{Lem}
\label{Lem:rBounds}
Assume $R<\frac1c$, with $c$ defined in \eqref{eq:drhBound}. Let $\gamma_\delta$ be a family of geodesics as in Theorem \ref{thm:radialpart}(a). Then
\begin{equation}
f(\delta)e^{-c(r_\delta-\delta)}\leq \vert \eta_\delta\vert \leq f(\delta)e^{c(r_\delta-\delta)}.
\label{eq:etaBounds}
\end{equation}
\begin{equation}
(1-C\delta)\,\vert t\vert \leq r_{\delta}(t)\leq \vert t\vert+\delta
\label{eq:rBoundsSemiWarped}
\end{equation}
holds for all time $t$ and all $\delta$.
Here $C=ce^{cR}$.
\end{Lem}
\begin{proof}[Proof of Lemma \ref{Lem:rBounds}]
 We will drop all $\delta$-subscripts, writing $r=r_{\delta}$ and so on.
 We may assume $t>0$. Then $\theta>0$ since $\theta(0)=0$ and $\thetadot>0$ by \eqref{eqn:thetadot}. 

We first prove \eqref{eq:etaBounds}. 
We rewrite \eqref{eqn:eta variation est} as $\bigl|\frac d{dt}\log|\eta|\bigr| \leq c \sin\theta = c\rdot$, so
$$ -c\rdot \leq \frac d{dt}\log|\eta| \leq c\rdot\,. $$
We integrate this from $t=0$ and exponentiate.
Using  $r(0)=\delta$ and $|\eta(0)| = f(\delta)$ (from \eqref{eqn:unit speed theta} since $\theta(0)=0$) we get \eqref{eq:etaBounds}.

\medskip

We now prove \eqref{eq:rBoundsSemiWarped}. 
The unit speed condition implies $\vert \rdot(t)\vert\leq 1$, so
\[r(t)-\delta=r(t)-r(0)=\int_0^t \rdot(s)\, ds\leq t.\]
This gives the upper bound in \eqref{eq:rBoundsSemiWarped}

For the lower bound we first consider the warped product case, for which $c=0$, hence $C=0$, so we need to show
\begin{equation}
\vert t\vert\leq r_{\delta}(t).
\label{eq:rBoundsWarped}
\end{equation}
 Differentiating $f(u)=f(r)\sin\theta$ in $t$ we get
\begin{equation}
 \label{eqn:udot0}
 \udot f'(u) = \rdot f'(r)\sin\theta + \thetadot f(r)\cos\theta
\end{equation}
which by virtue of $\rdot=\sin\theta$, $\thetadot = \frac{f'(r)}{f(r)}\cos\theta$ simplifies to 
\begin{equation}
 \label{eqn:udot}
\udot f'(u) = f'(r)
\end{equation}
Now $f(u)\leq f(r)$ by \eqref{eqn:def rho u}, so $u\leq r$ since $f$ is strictly increasing. Then $f'(u)\leq f'(r)$ since $f'$ is increasing. So \eqref{eqn:udot} implies
\[ \udot\geq 1\,.\]
From $u(0)=0$ we now get $t\leq u(t)\leq r(t)$ as claimed.
\medskip

Now we prove the lower bound in \eqref{eq:rBoundsSemiWarped} for the generalized warped product case.
In fact, we will prove the stronger lower bound
\begin{equation}
 \label{eq:rBoundsimproved}
 \left(1-C\frac{f(\delta)}{f'(\delta)}\right)|t| \leq  r_\delta(t)
\end{equation}
which implies \eqref{eq:rBoundsSemiWarped} by \eqref{eqn:f' f estimate}.

We use \eqref{eqn:udot0} again, but now we have the extra summand $-\frac{\partial_r|\eta|}{|\eta|}\cos\theta$ in $\thetadot$ (see \eqref{eqn:ham r theta}), which equals $-\frac{\partial_r|\eta|}{f(r)}$ by \eqref{eqn:unit speed theta}.
Therefore, \eqref{eqn:udot} is replaced by
\begin{equation}
\label{eqn:udot2} 
\notag
 \udot f'(u) = f'(r) - \partial_r|\eta|\,\cos\theta = f'(r)\left[1-\frac{\partial_r|\eta|}{f'(r)} \cos\theta\right]
\end{equation}
This implies \eqref{eq:rBoundsimproved} by the same argument as in the warped case, provided
\begin{equation}
\notag
1-\frac{\partial_r|\eta|}{f'(r)} \cos\theta \geq 1-C\frac{f(\delta)}{f'(\delta)}\,.
\end{equation}
Finally, this inequality follows from  $f'(r)\geq f'(\delta)$ (since $r\geq\delta$ and $f'$ is increasing), $\cos\theta\leq 1$, and 
$\partial_r|\eta| \leq C f(\delta)$. The latter inequality follows from  \eqref{eqn:error est} and \eqref{eq:etaBounds}, with $C=c e^{cR}$.
\end{proof}

\begin{proof}[Proof of Theorem \ref{thm:radialpart}]

By symmetry, it suffices to consider the $t>0$ part of the geodesics. Then $\theta>0$ since $\theta(0)=0$ and $\thetadot>0$.
We write $I_{\delta}=(T_-^{\delta},T_+^{\delta})$. By Theorem \ref{thm:two kinds} part (b) and (c), we have $r_{\delta}(T_{\pm}^{\delta})\coloneqq \lim_{t\to T_{\pm}^{\delta}} r_{\delta}(t)=R$. By Lemma \ref{Lem:rBounds}, we therefore have
\[(1-C\delta)\vert T_{\pm}^{\delta}\vert\leq R\leq \vert T_{\pm}^{\delta}\vert+\delta,\]
and sending $\delta\to 0$ gives 
\[T_{\pm}^{\delta}\xrightarrow{\delta\to 0} \pm R.\]  
Also $r_\delta(t)\to|t|$ follows from \eqref{eq:rBoundsSemiWarped}.

It remains to prove (b) in the warped product case. So let $\gamma=(r,y)$ and $\overline{\gamma}=(\overline{r},\overline{y})$ be two geodesics with $\delta=r(0)$, $\overline{\delta}=\overline{r}(0)$ and $\delta< \overline{\delta}$. We write $u$ and $\overline{u}$ for the function defined in \eqref{eqn:def rho u} associated to $r$ and $\overline{r}$ respectively. By the Clairaut-like relation \eqref{eqn:Clairaut} and the definition of $u$ \eqref{eqn:def rho u} we have
\begin{equation}
f(r)^2=f(u)^2+f(\delta)^2
\label{eq:u-r-relation}
\end{equation}
and similarly for $f(\overline{r})$. 
Since $\partial_r h=0$, we may use \eqref{eqn:udot} to find
\begin{equation}
\frac{d}{dt}f(u)=f'(r).
\label{eq:fu_iComp}
\end{equation}
and similarly for $f(\overline{u})$. 
The convexity of $f$ then implies 
\[\frac{d}{dt}\left(f(\overline{u})-f(u)\right)\geq 0\]
for all $t$ such that $\overline{r}(t)\geq r(t)$. Assume for contradiction that there is a maximal $t_0<T_{\delta}$ such that $\overline{r}(t)>r(t)$ on $[0,t_0)$. By continuity, $r(t_0)=\overline{r}(t_0)$. By \eqref{eq:fu_iComp}, we therefore have
\[f(\overline{u}(t_0))-f(u(t_0))=\int_0^{t_0} f'(\overline{r}(t))-f'(r(t))\, dt\geq 0.\]
Inserting this into \eqref{eq:u-r-relation}, we find
\[f(\overline{r}(t_0))^2-f(\overline{\delta})^2=f(\overline{u}(t_0))^2\geq f(u(t_0))^2=f(r(t_0))^2-f(\delta)^2\]
which rearranges to yield
\[f(\overline{r}(t_0))^2-f(r(t_0))^2\geq f(\overline{\delta})^2-f(\delta)^2>0.\]
This shows $\overline{r}(t_0)>r(t_0)$, contradicting the maximality of $t_0$.

\end{proof}
\begin{Rem}
The proof of the comparison theorem might seem a bit indirect, using $f(u(t))$ instead of $r(t)$. The reason being that it is in general false that $\overline{r}(t)-r(t)$ is increasing. Indeed, for the conical case $f(r)=r$, one can solve the equation of motion explicitly to find $r(t)=\sqrt{t^2+\delta^2}$, and $\overline{r}(t)-r(t)$ is strictly decreasing but never 0 when $\overline{\delta}>\delta$. The auxiliary function $f(u(t))$ is in this conical case simply $f(u(t))=t=f(\overline{u}(t)$, so $f(\overline{u}(t))-f(u(t))=0$ is increasing.
\end{Rem}


\begin{proof}[Proof of the first part of Theorem \ref{thm:winding}]
We will use the coordinates $\rho=f(r)$ and $\sin\theta=\xi$ again. We will drop the $\delta$-subscripts on all the variables. Since \eqref{eqn:thetadot} says $\dot{\theta}>0$, we can and will parametrise using $\theta$ instead of $t$.   From the equation of motion $\eqref{eqn:ham phi eta}$, and the Clairaut-like relation \eqref{eqn:Clairaut}, we have 
\begin{equation}
\vert \ydot\vert=\frac{\cos(\theta)}{\rho}
\label{eq:normphi'}
\end{equation}
(we write $\vert \ydot\vert = \vert \ydot\vert_h := \vert \ydot\vert_{h_{r_\delta(t)}}$).

Recall that $F$ is the inverse function of $f$. Using  the equation of motion \eqref{eqn:ham r theta}, $\partial_r = f'(r)\partial_\rho$ and $\frac1{f'(r)}=F'(\rho)$ we can write the measure as
\begin{equation}
\vert \ydot \vert  dt=\frac{F'(\rho)}{1-\rho \partial_\rho \log \vert \eta\vert)}\, d\theta\,.
\label{eq:Measure}
\end{equation}

We again start by assuming $h$ does not depend on $r$, since the argument becomes more transparent in this case.
In this case the formula \eqref{eq:Measure} simplifies to the elegant
\[\vert \ydot \vert  dt =F'(\rho)\,d\theta,\]
and one can use the Clairaut-like relation \eqref{eqn:Clairaut} to find
\[\rho=\frac{\rho_0}{\cos(\theta)},\]
where $\rho_0\coloneqq f(\delta)$.
Hence
\begin{equation}
f'(\delta) \vert \ydot \vert  dt=\frac{1}{F'(\rho_0)}\vert \ydot\vert  dt=\cfrac{F'\left(\cfrac{\rho_0}{\cos\theta}\right)}{F'(\rho_0)}\, d\theta.
\label{eq:LengthMeasureWarped}
\end{equation}
When $f$ is convex, $F'$ is decreasing, and the right hand integrand is bounded by $1$. So by the dominated convergence theorem and the non-oscillation condition \eqref{eqn:F condition},
\[\lim_{\delta\to 0} f'(\delta)\int_{T_-}^{T_+} \vert \ydot \vert  dt=\int_{-\pi/2}^{\pi/2} \lim_{\rho_0\to 0} \cfrac{F'\left(\cfrac{\rho_0}{\cos\theta}\right)}{F'(\rho_0)}\, d\theta=\int_{-\pi/2}^{\pi/2} \mathfrak{F}\left(\frac{1}{\cos\theta}\right)\, d\theta=C_f.\]
This proves part (a) when $h$ is independent of $r$. 
For the general case, we return to \eqref{eq:Measure}. Using the bounds \eqref{eqn:error est} again results in
\[\frac{F'(\rho)}{1+c\rho F'(\rho)}\, d\theta\leq \vert \ydot\vert dt\leq \frac{F'(\rho)}{1-c\rho F'(\rho)}\, d\theta\,.\]
Since $F$ is concave and $F(0)=0$, we have $\rho F'(\rho)\leq F(\rho)$, thus
\begin{equation}
\frac{F'(\rho)}{1+c F(\rho)}\, d\theta\leq \vert \ydot\vert dt\leq \frac{F'(\rho)}{1-c F(\rho)}\, d\theta\,.
\label{eq:ydotBounds}
\end{equation}
We therefore need bounds on both sides of \eqref{eq:ydotBounds}. Recall that $\rho$ is a function of $\delta$ and $\theta$. We consider its behaviour as $\delta\to0$ for fixed $\theta$.
The bounds \eqref{eq:etaBounds} say, since $|\eta|=\rho\cos\theta$,
\begin{equation}
 \label{eqn:rho ineq}
\frac{f(\delta)}{\cos\theta} e^{-c({r}-\delta)}\leq {\rho}\leq  \frac{f(\delta)}{\cos\theta} e^{c({r}-\delta)},
\end{equation}
so 
\[\lim_{\delta\to 0} {\rho(\theta)}=0\]
for every $\theta\in(-\frac\pi2,\frac\pi2)$. But then also ${r}=F({\rho})\xrightarrow{\delta\to 0}0$. Hence, using \eqref{eqn:rho ineq} and $f(\delta)=\rho_0$ we get
\[\lim_{\delta\to 0} \frac{{\rho}}{\rho_0}\cos\theta=1.\]
 Since, per assumption,
\[\lim_{\epsilon\to 0} \frac{F'(\epsilon \sigma)}{F'(\epsilon)}=\mathfrak{F}(\sigma)\]
and the convergence is uniform on compacts, we have
\[\lim_{\delta\to 0} \frac{F'({\rho})}{F'(\rho_0)(1\pm cF({\rho}))}=\lim_{\delta \to 0} \frac{F'\left( \frac{{\rho}}{\rho_0} \cos\theta \cdot \frac{\rho_0}{\cos\theta}\right)}{F'(\rho_0)}=\mathfrak{F}\left(\frac{1}{\cos\theta}\right).\]
Both the upper and lower bounds in \eqref{eq:ydotBounds} are bounded functions, hence integrable. By the dominated convergence theorem again, we find
\[\int_{-\pi/2}^{\pi/2} \mathfrak{F}\left(\frac{1}{\cos\theta}\right)\, d\theta\leq \lim_{\delta\to 0} f'(\delta) \ell(y_{\delta})\leq \int_{-\pi/2}^{\pi/2} \mathfrak{F}\left(\frac{1}{\cos\theta}\right)\, d\theta. \]
This proves part (a) in general.
\end{proof}
We now turn to the second part of Theorem \ref{thm:winding}. We will rescale the time $\tau$ according to \eqref{eq:dtau},
\begin{equation}
 \frac{d\tau}{dt}=\vert \dot{y}_{\delta}\vert=\frac{\vert \eta\vert}{f(r(t))^2}. 
 \notag
 \end{equation}
 subject to $\tau(0)=0$.
  We also introduce the rescaled angular momentum variable
\begin{equation}
 \overline{\eta}(t)\coloneqq \frac{\eta(t)}{f(\delta)},
 \label{eq:etaRescaling}
 \end{equation}
 which by \eqref{eqn:ham phi eta} implies $\overline{\eta}^{\#}(0)=f(\delta)\ydot(0)$. 
We recall our convention for the time-rescaled variables,
 \begin{align*}
 \tilde{y}(\tau(t))&\coloneqq y(t), &\tilde{\eta}(\tau(t))\coloneqq\overline{\eta}(t).
\end{align*}  
For most of the proof, we will not write out the $\delta$-subscript on $y,\tilde{y},\eta,\tilde{\eta}$.
 We will use a slash $'$ for $\tau$-derivatives and a dot $\dot{}$ for $t$-derivatives.
Recall that we assume 
$\tilde{\eta}(0)=\overline{\eta}(0)=f(\delta)\ydot(0)\to v_0$ as $\delta\to 0$.

\begin{proof}[Proof of the second part of Theorem \ref{thm:winding}]
By the unit speed parametrisation, the length of the time interval is the same as the length of the $Y$-component of the geodesic
\[\Delta\tau=\int_{T_-^{\delta}}^{T_+^{\delta}} \frac{d\tau}{dt}\, dt=\ell(y_\delta).\]
By the first part of Theorem \ref{thm:winding}, this has the asymptotic behaviour
\begin{equation}
\Delta\tau\sim \frac{C_f}{f'(\delta)}.
\label{eq:Deltatau}
\end{equation}

The equations of motion \eqref{eqn:ham phi eta}  get rescaled by $\frac{dt}{d\tau}$ and read
\begin{align}
\tilde{\eta}'(\tau)&=\frac{\dot{\eta}(t)}{f(\delta)\vert \dot{y}\vert}=-\frac{\partial_y \vert \tilde{\eta}\vert^2}{2\vert \tilde{\eta} \vert}\quad\qquad \tilde{y}' = \frac{\tilde{\eta}^\sharp}{|\tilde{\eta}^\sharp|}\,.
\label{eq:rescaledEoM}
\end{align}
We stress that $\tilde{\eta}$ is both a rescaled and time-reparametrised quantity.
 We first observe that the bound \eqref{eq:etaBounds} tells us
\[C'\leq \vert \tilde{\eta}\vert\leq C\]
uniformly in $\tau$ and $\delta$, where $C'>0$. Furthermore, since $h$ is $C^1$ in the $y$-directions and uniformly non-degenerate down to $r=0$, we can find a constant $\tilde{c}\geq 0$ such that
\[ \left\vert \partial_y \vert \tilde{\eta}\vert^2\right\vert \leq 2\tilde{c}\vert \tilde{\eta}\vert^2,\]
hence, by the unit speed condition and equation of motion \eqref{eq:rescaledEoM} respectively,
\[d( \tilde{y}_{\delta}(\tau),\tilde{y}_{\delta}(\sigma))\leq  \int_{\sigma}^{\tau} \left\vert\tilde{y}_{\delta}'(\upsilon)\right \vert\, d\upsilon\leq \vert \tau-\sigma\vert,\]
and
\[\vert \tilde{\eta}_{\delta}(\tau)-\tilde{\eta}_{\delta}(\sigma)\vert =\left\vert \int_{\sigma}^{\tau} \tilde{\eta}_{\delta}'(\upsilon)\, d\upsilon\right \vert\leq \tilde{C} \vert \tau-\sigma\vert.\]
This shows that the families $\tilde{y}_{\delta},\tilde{\eta}_{\delta}$ are uniformly (in $\tau$ and $\delta$) Lipschitz continuous. Let $K\subset \tilde{I}_{\delta}$ be any compact subset. By the Arzel\`{a}-Ascoli theorem, there is a subsequence $\delta_n\to 0$ such that $(\tilde{y}_{\delta_n}, \tilde{\eta}_{\delta_n})\to (\tilde{y}, \tilde{\eta})$ in $C(K;T^*Y)$.
Since we can rewrite the differential equations \eqref{eq:rescaledEoM} as integral equations, the limit $(\tilde{y}, \tilde{\eta})$ satisfies the same equations, with respect to the metric $h_0$ and  with initial condition $\tilde{y}(0)=y_0$ and $\tilde{\eta}(0)=v_0$. By the same argument, any subsequence of $(\tilde{y}_\delta,\tilde{\eta}_\delta)$ has a convergent subsequence, whose limit must be the same $(\tilde{y}, \tilde{\eta})$ since it satisfies the same initial conditions and same equations. By a standard argument, this implies the convergence of $(\tilde{y}_\delta,\tilde{\eta}_\delta)$ to $(\tilde{y}, \tilde{\eta})$. 
\end{proof}

\section{Discussion}
\label{Section:Discussion}

\subsection{The function $\mathfrak{F}$ and the constant $C_f$}
Recall that the constant $C_f$ in the length asymptotics of Theorem \ref{thm:winding} is given by 
\[
C_f=\int_{-\pi/2}^{\pi/2} \mathfrak{F}\left(\frac{1}{\cos(\vartheta)}\right)\, d\vartheta
\quad\text{ where }\ \mathfrak{F}(\sigma)=\lim_{\epsilon\to 0} \frac{F'(\sigma\epsilon)}{F'(\epsilon)}\ \text{ for } \sigma\geq1
\]
with $F$ the inverse function of $f$, see Assumption \ref{ass:non-oscill}. We now discuss properties of $\frakF$ and $C_f$ and their dependency on $f$.

It is useful to introduce a new variable $z$ defined by $x=e^{-z}$, so that $x\to0$ corresponds to $z\to\infty$.
We do this change of variable in both the domain and the range of $f$. That is, for given $f$ we consider the conjugated function
$$ \phi \coloneqq \expbar^{-1} \circ f \circ\expbar\,,\quad \expbar(z) = e^{-z} $$
or explicitly $\phi(z) = - \log f(e^{-z})$. For the examples in Remark \ref{rem:sing types} we get, with $\Phi$ the inverse function of $\phi$:
\begin{equation}
 \label{eqn:examples table}
\begin{array}{c @{\hspace{1cm}} | @{\hspace{1cm}} c @{\hspace{1cm}} |@{\hspace{1cm}} c @{\hspace{1cm}}|@{\hspace{0.1cm}} c}
f(x) & \phi(z) & \Phi(w) & \text{parameter range}\\
\hline
x^\alpha & \alpha z & \frac1\alpha w & \alpha \geq 1 \Tstrut\\[2mm]
\exp\left(-\ln\left(\frac{1}{x}\right)^{\mu}\right) & z^\mu & w^{1/\mu} & \mu > 1\\[2mm]
\exp\left(-\frac{1}{x^\beta}\right) & e^{\beta z} & \frac1\beta\log w & \beta > 0
\end{array}
\end{equation}
This displays the relative rates of vanishing as $x\to0$ of $f$  very clearly since they correspond to rates of growth of $\phi$ as $z\to\infty$: the rate increases with the parameter in each row, and from top to bottom.

\begin{Lem}\label{lem:Cf-new}
\begin{enumerate}
 \item[(a)]
Assume $f$ satisfies \eqref{eqn:props f} and $\mathfrak{F}$ exists. Then
\begin{equation}
\frac{1}{\sigma}\leq \mathfrak{F}(\sigma)\leq 1
\label{eq:Fbounds new}
\end{equation}
for all $\sigma\geq1$ and
\[2\leq C_f\leq \pi\,.\]
\item[(b)]
Assume $f$, $\ftilde$ satisfy \eqref{eqn:props f} and $\frakF$, $\frakFtilde$ exist.  Then
\begin{equation}
 \label{eqn:frakF comparison}
 f = O(\ftilde) \text{ near zero } \ \Longrightarrow\  \frakF \leq \frakFtilde\,,\ \text{ hence } C_f \leq C_{\ftilde} 
\end{equation}
\end{enumerate}
\end{Lem}

\begin{proof}
(a)  Since $f$ is increasing and convex, $F$ is increasing and concave, so $F'$ is positive and decreasing. This implies $\frakF\leq1$.
By L'H\^{o}pital's rule, 
\begin{equation}
 \label{eqn:limit frakF new}
\lim_{\epsilon\to 0} \frac{F(\sigma\epsilon)}{F(\epsilon)}=\sigma\lim_{\epsilon\to 0} \frac{F'(\sigma\epsilon)}{F'(\epsilon)}=\sigma\mathfrak{F}(\sigma).
\end{equation}
Now $F$ is increasing and $\sigma\geq1$, so 
$\frac{F(\sigma\epsilon)}{F(\epsilon)}\geq1$, which implies $\sigma\mathfrak{F}(\sigma)\geq1$.

Integration yields the inequalities for $C_f$.

(b)
 The inverse functions $F,\Phi$ of $f,\phi$ are related by $\Phi=\expbar^{-1} \circ F \circ\expbar$, or $F=\expbar\circ \Phi\circ\expbar^{-1}$, i.e.\ $F(\eps)=\exp(-\Phi(\log\frac1\eps))$. From this and \eqref{eqn:limit frakF new} we get for each $s\geq0$
$$ \lim_{w\to\infty}\left( \Phi(w+s) - \Phi(w) \right) = \log \left(\sigma \frakF(\sigma)\right)\,,\quad 
\sigma=e^s\,.
$$  
This means that $\Phi(w)$ has approximate linear growth as $w\to\infty$ at the rate $\frac1s\log \left(\sigma \frakF(\sigma)\right)$. Now let $f=O(\ftilde)$. Since the existence and value of $\frakF$ are not changed when multiplying $f$ by a constant, we may assume
$f\leq\ftilde$ near zero. This implies $\phi\geq\tilde{\phi}$, so $\Phi\leq\tilde{\Phi}$ near infinity. Assume for contradiction that $\frakF(\sigma)>\frakFtilde(\sigma)$ for some $\sigma>1$ (they are equal to one for $\sigma=1$). Then $\Phi$ would have a higher linear growth rate than $\tilde{\Phi}$, so $\Phi(w)>\tilde{\Phi}(w)$ for large $w$, contradicting $\Phi\leq\tilde{\Phi}$.

\end{proof}
For the examples in \eqref{eqn:examples table} we now get:
\begin{itemize}
 \item If $f(x)=x^\alpha$ then $\frakF(\sigma)= \sigma^{\frac1\alpha-1}$, 
hence 
$C_{f}= B(1-\tfrac1{2\alpha},\tfrac12)$
where $B$ is the Beta function.
 \item In particular, in the conical case $f(x)=x$ we have
 $\frakF(\sigma)= 1$, so $C_f= \pi$, and for \lq sharp cusps\rq, i.e.\ $f(x)=x^\alpha$ with $\alpha\to\infty$, we get $\frakF(\sigma)\to\frac1\sigma$, so $C_f\to2$.
 \item 
The monotonicity \eqref{eqn:frakF comparison} now implies that $\frakF(\sigma)=\frac1\sigma$, so $C_f=2$, for the examples in the second and third row of \eqref{eqn:examples table}.
\end{itemize}
\begin{Rem}
The monotonicity formula \eqref{eqn:frakF comparison} shows that $C_f=2$ for any worse singularity than the ones in \eqref{eqn:examples table}. For instance, define $\exp_{k+1}=\exp\circ \exp_k$ with $\exp_0(x)=x$. Let $k\geq 0$. Then 
\[f(x)=\exp\left(-\exp_k\left(\frac{1}{x}\right)\right)\]
satisfies \eqref{eqn:props f} and by \eqref{eqn:frakF comparison}, $C_f=2$ for any $k\geq 0$. 
\end{Rem}

\subsection{Example of $\mathfrak{F}$ failing to exist}
The idea is to add a fast but suitably small oscillation to $F=f^{-1}$, preventing convergence. We perturb a family $F(x)$ tending to $0$ as $x^{\alpha}$ for $\alpha\in (0,1)$.
\begin{Lem}
Let $\alpha\in (0,1)$ and $c\geq \frac{2-\alpha}{1-\alpha}\cdot \frac{1+\alpha}{\alpha}$. Define
\[F(x)=x^{\alpha}\left(c+\sin(\log(x))\right)= cx^\alpha +\tfrac1{2i} x^{\alpha+i}-\tfrac1{2i} x^{\alpha-i}.\]
Here $i=\sqrt{-1}$.
Then $F$ is positive, increasing and concave, but $\lim\limits_{\epsilon\to 0} \frac{F'(\epsilon\sigma)}{F'(\epsilon)}$ fails to exist unless $\log(\sigma)= 2n\pi$ for some $n\in \mathbb{N}_0$.
\label{Lem:FNonConv}
\end{Lem}
\begin{proof}
We compute
\[F'(x)= x^{\alpha-1}\left(\alpha c+\alpha\sin\log x+\cos\log x\right), \]
\[F''(x)= x^{\alpha-2}\left(\alpha(\alpha-1)c+(\alpha^2-\alpha -1)\sin\log x+(2\alpha-1)\cos\log x\right). \]
$F'(x)>0$ clearly holds if $c>\frac{1+\alpha}{\alpha}$, which is true since $c\geq \frac{2-\alpha}{1-\alpha}\cdot \frac{1+\alpha}{\alpha}$. Similarly, since $\alpha^2-\alpha-1<0$ and $|2\alpha-1|<1$,
\[F''(x)\leq x^{\alpha-2}\left(\alpha(\alpha-1)c-(\alpha^2-\alpha-1)+1\right)\leq 0\]
if $c\geq \frac{2-\alpha}{1-\alpha}\cdot \frac{1+\alpha}{\alpha}$.

To see that the limit fails to exist, assume it did. Then, by L'H\^{o}pital's rule, the limit $\lim\limits_{\epsilon\to 0} \frac{F(\epsilon\sigma)}{F(\epsilon)}$ would also exist. But
\begin{align*}
\frac{F(\epsilon\sigma)}{F(\epsilon)}=\sigma^{\alpha}\frac{c+\sin\log(\epsilon\sigma) }{  c+\sin\log\epsilon }=\sigma^{\alpha} \frac{c+\sin\left(\log\epsilon+\log(\sigma)\right)}{c+\sin\log\epsilon},
\end{align*}
and the existence of the limit as $\epsilon\to 0$ is equivalent to the existence of the limit
\[\lim_{x\to \infty} \frac{c+\sin(x+y)}{c+\sin(x)}\]
for all $y=\log(\sigma)\geq 0$. Consider the sequence $x_n=\frac{\pi}{2}+n\pi$. Then
\[\frac{c+\sin(x_n+y)}{c+\sin(x_n)}=\frac{c+(-1)^n \cos y}{c+(-1)^n},\]
and this converges only if $\cos y=1$, i.e.  $y=\log(\sigma) \in 2\pi \mathbb{N}_0$. The converse follows from
\[\frac{F'(\sigma\epsilon)}{F'(\epsilon)}=\sigma^{\alpha-1}\]
whenever $\log(\sigma)\in 2\pi \mathbb{N}_0$.
\end{proof}


We formulate a small question concerning $\mathfrak{F}$.
\begin{Quest}
Assume $f$ satisfies \eqref{eqn:props f}, and additionally, that
  \[\lim_{x\to 0} \frac{f(x)}{x^{\alpha}}=0\]
for all $\alpha\geq 1$. 

 Does 
$\mathfrak{F}$ exist and is $\mathfrak{F}(\sigma)=\frac{1}{\sigma}$? 

\end{Quest}

\subsection{Why convexity?} \label{subsec:why convexity}
One possible extension of our work is to relax the assumptions on $f$. One could for instance ask what happens when $f$ is concave instead of convex? What one will notice then is that the bounds on $\mathfrak{F}$ are no longer true, and $C_f$ might very well be infinite. When $f(x)=x^{\alpha}$ with $0<\alpha<1$, $\mathfrak{F}(\sigma)=\sigma^{\frac{1-\alpha}{\alpha}}$ is unbounded, and $C_f$ is infinite for $\alpha\leq \frac{1}{2}$.
The border case $\alpha=\frac{1}{2}$ can be explicitly solved in the warped case. The formula \eqref{eq:LengthMeasureWarped} holds for arbitrary positive and increasing $f$, hence also for $f(x)=\sqrt{x}$. Here $F(\rho)=\rho^2$, so
\[f'(\delta)\ell(y_{\delta})=\int_{-\Theta_{\delta}}^{\Theta_{\delta}} \frac{d\theta}{\cos\theta},\]
where $\cos\Theta_{\delta}=\sqrt{\frac{\delta}{R}}.$ The integral is explicitly computable, and the result is
\[f'(\delta)\ell(y_{\delta})=2 \log\left(\frac{R+\sqrt{R^2-\delta^2}}{\sqrt{\delta}}\right).\]
The essential point here is that the length behaves as 
\[f'(\delta)\ell(y_{\delta}) \sim \log\delta^{-1},\]
so the asymptotic behaviour has become more complicated.

\section{Metrics on isolated singularities}\label{sec:singularities}

Singular spaces typically arise in one of the following ways:
\begin{itemize}
 \item
as subsets of smooth manifolds $M$, e.g. as solution sets $\{x\in M:F(x)=0\}$ of (systems of) equations, where $F:M\to N$ is a smooth map to a manifold $N$ and singularities may arise at points $p$ where the differential $dF_{|p}$ is not surjective; an important class of examples are algebraic varieties;
\item  
as quotients of smooth spaces by non-free group actions; important classes of examples are orbifolds and moduli spaces.
\end{itemize}
These spaces are often equipped with natural metrics. For example, complex projective algebraic varieties carry metrics induced by the Fubini-Study metric, and the Riemannian moduli space carries a natural metric, which has (non-isolated) cuspidal singularities.

However, in this section we carefully distinguish non-metric and metric aspects of singularities. The discussion will be summarized in Subsection \ref{subsec:summary sing}.
\subsection{Singularities of type $[s]$}

We focus on singular spaces arising as subsets of manifolds and only consider isolated singularities which are described by a profile (or 'shrinking') function $s$ of a single variable, in the sense described below. First, we define:
\begin{Def}\label{def:singularity}
 Let $M$ be a manifold and $X\subset M$. Let $p\in X$.
We say that 
 $p$ is an \textbf{isolated singularity} of $X$ if $X_0=X\setminus\{p\}$ is a submanifold of $M$ and $p$ lies in the closure of $X_0$.
\end{Def}
  We define a \textbf{profile function} to be a $C^1$ function $s:[0,\eps)\to[0,\infty)$ for some $\eps>0$ which satisfies
 \[ s(0)=0\,,\ s(z)>0\text{ for }z>0\,.\]
 We call two profile functions $s,\stilde$ \textbf{equivalent} if $\stilde = as$
 near $0$, where $a$ is positive and $C^1$ on some interval $[0,\epsilon)$. Denote the equivalence class of $s$ by $[s]$.

Let us say that two curves $\gamma$, $\gammabar$ in $M$ with $\gamma(0)=\gammabar(0)=p$ and $\dot{\gamma}(0)\neq 0\neq \dot{\overline{\gamma}}(0)$ are \textbf{$s$-tangent at $p$} if $|\gamma(t)-\gammabar(t)|=O(s(t))$ near $t=0$ in one (hence any) local coordinate system. If $s(z)=z^k$, $k\in\N$, then this corresponds to tangency of order $k-1$. Roughly, we will say that $X$ has a singularity of type $s$ at $p$ if it is a union of curves which are pairwise $s$-tangent at $p$, but not tangent to higher order (i.e.\ also    $|\gamma(t)-\gammabar(t)|\geq cs(t)$ for a constant $c>0$).
\subsubsection{The cuspidal case}
We call a profile function $s$ \textbf{cuspidal} if $s'(0)=0$.
For a cuspidal profile function $s$ the \textbf{$s$-blow-down map} $\beta_s$ is defined by
\begin{equation}
\label{eqn:def betas} 
 \beta_s:[0,\eps)\times  \R^{N-1}\to [0,\eps)\times \R^{N-1}\,,\quad (z,u) \mapsto (z,s(z)u)
\end{equation}
It shrinks the hyperplane at height $z$ by the factor $s(z)$. In particular the boundary plane $z=0$ is collapsed to the origin.
\begin{Def}\label{def:cusp singularity}
 Let $M$ be a manifold and $X\subset M$. Let $p\in X$, and let $s$ be a cuspidal profile function. We say that $X$ has a \textbf{cuspidal singularity of type $[s]$} at $p$ if there are coordinates $(z,x):U\to\R\times\R^{N-1}$ on a neighbourhood $U\subset M$ of $p$, mapping $p$ to the origin, in terms of which
\begin{equation}
 \label{eqn:sing of type s}
X\cap U=\beta_s(\Xtilde)\,,\quad \Xtilde \coloneqq  \bigcup_{z\in[0,\eps)} \{z\} \times Y_z
\end{equation}
where $\beta_s$ is the $s$-blow-down map 
and   $Y_z\subset\R^{N-1}$ are closed submanifolds varying smoothly\footnote{That is,  $Y_z=\iota(Y,z)$ for a fixed manifold $Y$, with $\iota:Y\times[0,\eps)\to\R^{N-1}$ smooth and
$\iota(\cdot,z)$ an embedding for each $z$. Equivalently, $\Xtilde$ is a p-submanifold of $\R^{N-1}\times [0,\infty)$, see \cite{Mel:DAMWC} or \cite{Gri:BBC}.}
with $z\in [0,\eps)$. We call $\Xtilde$ a \textbf{resolution} of $X$.

The curve $z\mapsto (z,0)$ is called  \textbf{axis} of the singularity. We also call the germ of $X\subset M$ at $p$ a singularity of type $[s]$.
\end{Def}

These singularities are natural generalizations of surfaces of rotation in $\R^3=\R\times\R^2$ with profile function $s$, for which $Y_z=S^1\subset\R^2$ for each $z$ (see also Example \ref{ex:surf revol}). Compared to isolated singularities of algebraic sets they allow more flexibility in that $s$ can vanish to infinite order at zero, but are also more restricted since the blow-down map shrinks each point $u\in Y_z$ by the same factor $s(z)$. For example, the set $\{(z,x,y):\, (\frac xz)^2 + (\frac y{z^2})^2 = 1\,,\ z>0\}\cup\{0\}\subset\R^3$ is not of this type. See \cite{GraGri:GSS} for more on this.

The singularity type is an equivalence class $[s]$ because equivalent profile functions define the same class of singularities: if $\stilde =as$ then $s, (Y_z)$ define the same $X$ as $\stilde, (a(z)^{-1}Y_z)$.

It follows from the cusp condition $s'(0)=0$ that a cuspidal singularity $(X,p)$ has a well-defined \textbf{tangent direction} at $p$, which is a ray in $T_pM$. This is the set of tangents at $p$ of curves in $X$ starting at $p$, and is equal to the tangent cone in the sense of \cite{BerLyt:TSGHLSS}, for instance.

Note that this notion of singularity type is purely differential, i.e.\ there is no metric involved. It is also natural in the sense that if the condition holds in one coordinate system then it holds in any other having the same axis.\footnote{We don't make use of this. Here is a sketch of the proof:
(We assume here for simplicity that everything is smooth.)
We need to show that any diffeomorphism $\Phi:\R^{N}\to\R^N$ which fixes the $z$-axis pointwise lifts to a diffeomorphism under $\beta_s$, locally near zero. By the argument in \cite[Proof of Lemma 2]{Mel:RBIASS} this reduces to 
 showing that the vector fields generating these diffeomorphisms lift to smooth vector fields under $\beta$. These vector fields are spanned over $C^\infty$ by the vector fields $x_i\partial_{x_j}$ and $x_i\partial_z$ for $i,j=1,\dots,N-1$. A simple calculation shows that $\beta_*(U_{ij})= x_i\partial_{x_j}$ for $U_{ij}=u_i\partial_{u_j}$ 
 and $\beta_*(Z_i)= x_i\partial_z$ for $Z_i=s u_i \partial_z - s'u_i \sum_j u_j\partial_{u_j}$, so $U_{ij}$, $Z_i$ are the desired smooth lifts. 
 
 A fine point is that the axis is only determined by $X$ up to a perturbation of order $s$.
}

\subsubsection{Conical singularities} \label{subsubsec:conical sing}
 Conical singularities can be defined in essentially the same way, where the profile function is $s(z)=z$ or equivalent to this. However, while cuspidal singularities always lie in a half space since they have a unique tangent direction at $p$, it would be unnatural to require this for a cone. 
 Therefore, we \textbf{define a conical singularity} in the same way as in Definition \ref{def:cusp singularity}, except that we replace the map $\beta_s$ above by the polar coordinates map
\begin{equation}
\label{eqn:beta conical}
 \betaconic: [0,\infty) \times \S^{N-1}\to\R^N\,,\quad (z,\omega)\mapsto z\omega 
\end{equation}
and take  a smooth family of closed submanifolds $Y_z\subset\S^{N-1}$. 
The space $[0,\infty) \times \S^{N-1}$ is called the \textbf{blow-up} of $\R^N$ at the point $0$ and $\betaconic$ the blow-down map, and the definition can be restated by saying that $X$ is resolved by blowing up $p$. The map $\beta_s$ in \eqref{eqn:def betas}, with $s(z)=z$,  then is $\betaconic$ written in one of the projective coordinate systems. See \cite{Mel:DAMWC} or \cite{Gri:BBC} for details on this.

If $s(z)=z^k$ then the map $\beta_s$ is the blow-down map for a quasi-homogeneous blow-up (in projective coordinates; see \cite{Gri:SBQC}, \cite{KotMel:GBCFP}), so for general $s$ we have defined a generalized notion of quasi-homogeneous blow-up.

The other remarks above carry over to the conical case. In particular, this notion of conical singularity is  differential, not metric.

\subsection{Metrics on singularities of type $[s]$}
\label{subsec:bl-up}

 It is natural to consider metrics on $X$ which are induced by smooth metrics on $M$.
 
\begin{Def} \label{def:induced distance}
Let  $(M,g_M)$ be a smooth Riemannian manifold and $X\subset M$.
Assume that $X$ has an isolated singularity at $p\in X$. The \textbf{induced metric}  for $X$, denoted $g_X$, is the Riemannian metric on $X_0=X\setminus\{p\}$ obtained by restriction of $g_M$ to $TX_0$. 
The \textbf{induced distance} on $X$, denoted $d_X$, is the distance function on $X$ defined by $g_X$.
\end{Def}
That is, if  $q,q'\in X$ then $d_X(q,q')=\inf \ell(\gamma)$ where the infimum is taken over all curves $\gamma:(0,1)\to X_0$ with $\lim_{t\to0}\gamma(t)=q$, $\lim_{t\to1}\gamma(t)=q'$.
This is sometimes called the intrinsic distance on $X$, and is to be distinguished from the extrinsic distance, which is the restriction of the distance function $d_M$ (defined on $M\times M$ by $g_M$) to $X\times X$.

To understand the geometry of $(X,g)$, e.g.\ the behaviour of geodesics near $p$, it is useful to have a normal form of the metric. Specifically, we ask under which conditions the induced distance makes $X$ a Riemannian space with conical or cuspidal singularity at $p$ as in Definition \ref{def:cone cusp}. This turns out to be rather subtle even for spaces with cuspidal or conical singularity as defined above. We discuss this now.

\subsubsection{Warped products}
In order to get a warped product metric on $X$ (rather than just a generalized one) one expects to have to impose rather rigid conditions. We will consider the case where $M=\R^N$ with the Euclidean metric, and where the resolution $\Xtilde$ (defined using standard coordinates $(z,x)$ on $\R^N$) is a product $[0,\eps)\times Y$.

\begin{Prop} \label{prop:warped example}
Let $X\subset \R^N$ have an isolated singularity at $0$. Let $g_X$ be the metric on $X_0$ induced by the Euclidean metric on $\R^N$, and $d_X$ the induced distance.
\begin{enumerate}[label=(\alph*)]
 \item (Cuspidal case)
 If $X$ has a cuspidal singularity of type $[s]$, and if, in \eqref{eqn:sing of type s},

\begin{equation}
 \label{eqn:cond warped cusp}
 Y_z=Y \ \forall z\ \text{ where } Y\subset \R^{N-1} \text{ is contained in the unit sphere } \S^{N-2} \text{ centered at }0
\end{equation}
then $g_X$ is a warped product metric. The warping function is determined by
\begin{equation}
 \label{eqn:def r}
  f(r(z))=s(z)\ \text{ where }\ r(z)=\int_0^z \sqrt{1+(s'(w))^2}\,dw\,,
\end{equation}
and $h$ is the metric on $Y$ induced by the Euclidean metric on $\R^{N-1}$.
Also, $f(r)\sim s(r)$ as $r\to0$, and $f$ is convex iff $s$ is convex, so $(X,d_X)$ is a Riemannian space with an isolated cuspidal singularity in this case.
\item (Conical case)
If $X$ has a conical singularity, and if
\begin{equation}
 \label{eqn:cond warped cone}
 Y_z=Y \ \forall z \ \text{ for some $Y\subset \S^{N-1}$}
\end{equation}
then $g_X$ is a warped product metric with warping function $f(r)=r$, so $(X,d_X)$ is a Riemannian space with an isolated conical singularity. The metric $h$ is the metric on $Y$ induced by the standard metric on $\S^{N-1}$.
\end{enumerate}
\end{Prop}
Note that in the cuspidal case there is an additional condition, $Y\subset \S^{N-2}$, on the cross section $Y$, while in the conical case there isn't. See also the remark after Theorem \ref{thm:gragri cuspidal}. 
Condition \eqref{eqn:cond warped cone} says that $X$ is a \lq straight\rq\ cone with an arbitrary base $Y$.

The idea of the proof of part (a) is that $r(z)$ is arc length along the curve $z\mapsto (z,s(z)u)$ for any fixed $u$ with $|u|=1$. The unit sphere condition is needed to ensure that mixed terms vanish, i.e.\ that the curves $u= \const.$ are orthogonal to the level sets $z= \const$.

\begin{proof}
\begin{enumerate}[label=(\alph*)]
 \item
  We calculate the induced Riemannian metric on $X_0=X\setminus\{0\}$. 
 First, we write the Euclidean metric on $\R^N$  in coordinates $z,u$, i.e.\ we pull it back under the map $\beta_s$, see \eqref{eqn:def betas}. The differential of $s(z)u_i$ is
 $$ d(s(z)u_i) = s(z)\,du_i + s'(z)u_i\,dz \,,$$
so we get
\begin{align*}
 \beta_s^*(\geucl) &= dz^2 + \sum_{i=1}^{N-1} \left(d(su_i)\right)^2 \\
 &= (1+(s')^2 |u|^2)\,dz^2 + 2 ss' dz \sum_{i=1}^{N-1}u_i\,du_i + s^2 \sum_{i=1}^{N-1} du_i^2 
\end{align*}
where $s=s(z)$ etc.
Now we restrict to $Y$. From $Y\subset \S^{N-2}$ we have that $|u|^2$ is constant equal to 1 on $Y$. In particular, the mixed term, which is $ss'\,dz\,d(|u|^2)$, vanishes when pulled back to $Y$, so we get
$$ g_{X_0} = (1+(s')^2)\,dz^2 + s^2 \,h\,. $$
 With $r$ defined in \eqref{eqn:def r} we have $dr = \sqrt{1+(s')^2} \,dz$ and $s(z)=f(r)$. So in $r,u$ coordinates the metric is $dr^2 + f(r)^2 \,h$ as claimed.

From $s'(0)=0$ we get $r(z)\sim z$ as $z\to0$, hence $f(r)\sim s(r)$ as $r\to0$. Finally, one calculates
$$ f'(r(z)) = \frac{s'(z)}{\sqrt{1+(s'(z))^2}} $$
and since the function $z\mapsto r(z)$ as well as the function $\sigma\mapsto \frac\sigma{\sqrt{1+\sigma^2}}$ (and hence also their inverses) are strictly increasing, it follows that $f'$ is increasing iff $s'$ is increasing, so $f$ is convex iff $s$ is convex.
\item
It is well-known that the Euclidean metric on $\R^N$ reads
$\betaconic^*(\geucl) = dr^2 + r^2 g_{\S^{N-1}}$ in polar coordinates (where we write $r$ instead of $z$ in \eqref{eqn:beta conical}). Since $\tilde{X}$ is $\beta^{-1}(X)=[0,R)\times Y$ in polar coordinates, it follows that $g=dr^2+r^2 h$.
\end{enumerate}
\end{proof}

\begin{Rem}
 \label{rem:more conical}
  Proposition \ref{prop:warped example} implies that our results on geodesics apply in the cuspidal case if \eqref{eqn:cond warped cusp} holds and $s$ is convex, and in the conical case if \eqref{eqn:cond warped cone} holds. 
  There are also other cases where one obtains conical warped product metrics:
 
\begin{enumerate}
 \item[(a)]
\textit{$X$ is as in Proposition \ref{prop:warped example}(a), except that $s'(0)>0$, and $s$ is convex.}
\\
Note in particular the additional condition $Y\subset\S^{N-2}$ in \eqref{eqn:cond warped cusp}. However, this condition is not needed if $s(z)=cz$ with a constant $c>0$ since then $X$ can be expressed as in Proposition \ref{prop:warped example}(b) (with a different $Y$).\footnote{This can also be seen using a calculation as in the proof of Proposition \ref{prop:warped example}(a): when we set $r=\sqrt{1+c^2|u|^2}z$ then $dr^2$ will absorb both the $dz^2$ and the mixed term in $\beta_s^*(\geucl)$, even if $|u|^2$ is not constant on $Y$.
}
\\
Note that if $s'(0)>0$ and $s$ is $C^2$ then $f$ is also $C^2$, so by Remark \ref{rem:sing types} (c) one does not need to require convexity of $s$ then.
 \item[(b)]
 \textit{$X$ is as in Proposition \ref{prop:warped example}(a), except that $s(z)=z^{\alpha}$ with $\alpha\in (0,\frac23]$.}
\\
(Note that $s$ is not even differentiable at zero then.) The reason is that 
\[r(z)=\int_0^z \sqrt{1+\alpha^2 w^{2\alpha-2}}\, dw=\sum_{k=0}^{\infty} \frac{c_k \alpha^{1-2k}}{2k(1-\alpha)+\alpha}z^{2k(1-\alpha)+\alpha}=z^\alpha+\beta z^{2-\alpha}+\dots,\]
where $c_k$ are the coefficients of the Taylor series of $\sqrt{1+x}$ at $x=0$, so $\beta=\frac1{2\alpha(2-\alpha)}>0$. This implies $f(r)=r-\beta r^{\gamma}+\dots$ where $\gamma=\frac{2-\alpha}\alpha$, which is conical but not convex. However, if $\alpha\leq\frac23$ then it is $C^2$ down to $0$, so $f$ is equivalent to a convex function.
\\
For example, if $s(z)=\sqrt z$ and $Y=\S^{N-2}$ then $X$ is in fact smooth at $p$.
\end{enumerate}

\end{Rem}
\begin{Ex} \label{ex:surf revol}
If $X\subset\R^3$ is a surface of revolution with profile function $s$ then it has a cuspidal warped product metric if $s$ is cuspidal and a conical warped product metric if $s$ is $C^1$ at zero and $s'(0)>0$, or if $s(z)=z^\alpha$ with $\alpha\in(0,\frac23]$.
\end{Ex}

\subsubsection{Generalized warped products} \label{section: non-warped example} 
 If we use more general metrics on $M$ near $p$ or do not impose the product type condition $Y_z=Y\,\forall z$ then we cannot expect the induced metric on $X$ to have warped product structure. We discuss what is known in the conical and $k$-cuspidal case, i.e. $f(r)=r^k$ for some $k\in\N$, $k\geq1$.

First,  in the conical case one always gets a generalized warped product, as shown by Melrose and Wunsch.
\begin{Thm}[\cite{MeWu:Geo}]\label{thm:MeWu}
 Let $X\subset M$ have a conical singularity as defined above. Let $g_M$ be a Riemannian metric on $M$ and $d_X$ the induced distance function on $X$.
 
 Then $(X,d_X)$ is a Riemannian space with an isolated conical singularity at $p$, as in Definition \ref{def:cone cusp}.
\end{Thm}

Theorem \ref{thm:MeWu} is proven in two steps: First one uses normal polar coordinates in $M$, centred at $p$, to write the induced metric as
$g = dr^2 + r^2 h$ on $\tilde{X}=[0,R)\times Y$, where all $Y_z$ are identified with a fixed $Y$ and $h$ is a smooth 2-tensor on $\tilde{X}$ restricting to a metric on $\{0\}\times Y$. However, this is more general than \eqref{eq:Metric} since $h$ may involve $dr^2$ terms and also mixed  $dr\, dy$ terms.
In a second, more difficult step one proves that one can change coordinates so that these terms are removed, i.e.\ $g$ takes the form \eqref{eq:Metric}. For this one shows that for each $q\in Y$ there is a unique geodesic in $(0,R)\times Y$ hitting the boundary at $(0,q)$, and that these geodesics together define a fibration $\Phi:[0,R')\times Y\to U$ of a neighbourhood $U\subset\tilde{X}$ of $r=0$. Then $\Phi$ is the desired coordinate system.
\medskip

The $k$-cuspidal case was analysed by Grandjean and Grieser in \cite{GrGr18}, with more detail and a correction given in \cite{BeyGri:IGIC}.
Let $X$ be given as in \eqref{eqn:sing of type s} with $s(z)=z^k$, $k\geq2$, and let $g_M$ be any Riemannian metric on $M$. First, by \cite[Proposition 7.3]{GrGr18} and \cite{BeyGri:IGIC} a trivialization $U\cong [0,r_0)\times\dXtilde$ of a neighbourhood $U$ of $\dXtilde$ can be chosen so that  the induced metric on $X$, pulled back to $\Xtilde$, has the form
\begin{equation}
 \label{eqn:cusp metric general}
g =(1+S(z,u) z^{k})\,dz^2 +z^{2k}h
\end{equation}
for a smooth function $S$ and a smooth 2-tensor $h$ on $\Xtilde$ restricting to a metric $h_0$ on $\{0\}\times\partial\Xtilde=Y_0$.\footnote{In \cite[Proposition 7.3]{GrGr18} 
this was stated with $z^{2k-2}$ instead of $z^k$, which is a stronger claim in case $k\geq3$. However, the proof of that proposition makes the implicit assumption that the axis of the singularity is tangent to order $k-1$ to the geodesic in $M$ which starts at $p$ in the same direction as the axis. This is trivially satisfied if $k=2$ but need not be true if $k\geq3$. The proof of the correction and further explanation is given in \cite{BeyGri:IGIC}.}

This cannot be improved in general: Consider the subspace $\calN \subset C^\infty([0,r_0)\times\dXtilde)$ of functions $f$ satisfying $d_\dXtilde f = O(z^{k-1})$, and denote by $S_0$ the equivalence class of $S$ modulo $\calN$. For example, if $k=2$ then $S_0$ can be identified with the restriction of $S$ to $\dXtilde$ modulo constants, so it vanishes if and only if $S_{|\dXtilde}$ is constant.
By 
 \cite[Lemma 2.2]{GrGr18} and \cite{BeyGri:IGIC} the class $S_0$ is an invariant of the singularity and metric. In particular, if $S_0$ is non-zero then the metric can not be put in generalized warped product form by any choice of coordinates on $\Xtilde$. However, if $S_0$ vanishes then an argument using geodesics hitting $p$, similar to the argument in  \cite{MeWu:Geo}, shows:
\begin{Thm}[Theorem 1.2 + Remark 3.4 in \cite{GrGr18}] \label{thm:gragri cuspidal}
 Suppose $X$ has a $k$-cuspidal singularity as discussed above. If the equivalence class $S_0$ defined above vanishes then the metric on $X$ is a generalized warped product metric with warping factor $f(r)=r^k$.
\end{Thm}
For example, if $M=\R^N$ with the Euclidean metric then $S_0(z,u)=z^{k-2}|u|^2$, with $u$ restricted to  $Y_0\subset\R^{N-1}$. This shows that the additional condition on $Y$ in \eqref{eqn:cond warped cusp} is necessary, up to a constant factor.

In \cite{GrGr18} the geodesics hitting $\dXtilde$ were analysed also in the more general case where $S(z,u)=z^{k-2}\tilde S(u) + O(z^{k-1})$ with $\tilde S$ a Morse function on $\dXtilde$.
\subsection{Summary of Section \ref{sec:singularities}} \label{subsec:summary sing}
We distinguish differential and metric notions of conical/cuspidal singularities. 

The differential notion refers to subsets $X$ of a manifold $M$ with isolated singularity $p\in X$. It is based on the idea of $s$-tangency of curves, where $s$ is a profile function. 

The metric notion is based on the idea of generalized warped product metrics. It is related to the differential notion as follows. Given a Riemannian metric $g_M$ on $M$, the induced metric $g_X$ on $X$ is
\begin{itemize}
 \item conical if $X$ has a conical singularity in the differential sense
 \item cuspidal if $X$ has a cuspidal singularity in the differential sense and if it satisfies additional requirements in its relation to $g_M$, like the vanishing of $S_0$ in Theorem \ref{thm:gragri cuspidal}.
\end{itemize}
Thus, in the cuspidal case, the metric notion is more restrictive than the differential notion. 

We also mention that there are other notions of conical singularity in the literature. For example,  the notion of corner domain introduced by
Dauge in \cite{Dau:EBVPCDSAS} refers to subsets of $\R^N$ which arise from 'straight' conical spaces as in Proposition \ref{prop:warped example}(b) by local diffeomorphisms of the ambient space $\R^N$. This is more special than our notion of conical singularity given in Subsection \ref{subsubsec:conical sing}. (However, corner domains also include non-isolated singularities such as those arising from the base $Y$ having corner domain singularities itself.)

\section{Acknowledgements}
\changelocaltocdepth{1}
\subsection*{Funding}
The first author was partially supported by DFG Priority Programme 2026 ‘Geometry at Infinity’. The second author did not receive any grant for writing this article.

%
%

\end{document}